\documentclass[11pt,reqno,twoside]{amsart}
\usepackage{mathrsfs}
\usepackage{amsfonts,amssymb,amsmath,amsthm}
\usepackage{cite}
\usepackage[colorlinks=true,citecolor=red,linkcolor=blue]{hyperref}
\usepackage{titletoc}
\usepackage{geometry}
\usepackage{bm}
\usepackage{indentfirst}
\usepackage{graphicx}
\usepackage{stmaryrd}
\usepackage{tikz}
\usepackage{ulem}
\usepackage{color,soul}
\usepackage{setspace}
\usepackage[pagewise]{lineno}
\usepackage{float}
\usepackage{todonotes}

\def\R {\mathbb{R}}
\def\N {\mathbb{N}}
\def\Z {\mathbb{Z}}
\def\d{{\,\rm d}}
\numberwithin{equation}{section}
\newtheorem{theorem}{Theorem}[section]
\newtheorem{lemma}[theorem]{Lemma}
\newtheorem{corollary}[theorem]{Corollary}
\newtheorem{proposition}[theorem]{Proposition}

\newtheorem{remark}[theorem]{Remark}
\theoremstyle{definition}
\newtheorem{example}[theorem]{Example}
\newtheorem{definition}[theorem]{Definition}

\newcommand{\supp}{\operatorname{supp}}

\tikzstyle{idea} = [rectangle, rounded corners, minimum width=2cm, minimum height=1cm, text centered, draw=black, align=center]
\tikzstyle{process} = [rectangle, minimum width=3cm, minimum height=1cm, text centered, draw=black, align=center]
\tikzstyle{point} = [coordinate, on grid]
\tikzstyle{arrow} = [thick,->,>=stealth]
\tikzstyle{dasharrow} = [dashed,->,>=stealth]

\makeatletter
\@namedef{subjclassname@2020}{\textup{2020} Mathematics Subject Classification}
\makeatother

\title[Logarithmic Ultra-analyticity]{Sharp Logarithmic Ultra-analyticity for Fractional and Nonlocal Elliptic Equations}

\author[H. Dong]{Hongjie Dong}
\author[Y. Hu]{Yeyao Hu}
\author[M. Wang]{Ming Wang}

\address[H. Dong]{Division of Applied Mathematics, Brown University, 182 George Street, Providence, RI 02912, USA}
\email{Hongjie\_Dong@brown.edu}

\address[Yeyao Hu]{School of Mathematics and Statistics, HNP-LAMA, Central South University, Changsha, Hunan 410083, P.R. China}
\email{huyeyao@gmail.com}

\address[Ming Wang]{School of Mathematics and Statistics, HNP-LAMA, Central South University, Changsha, Hunan 410083, P.R. China}
\email{m.wang@csu.edu.cn}

\begin{document}

       \subjclass[2020]{35R11; 
       35B65; 
       35A20; 
        42B37.
        }
	
        \keywords{Fractional Laplacian, nonlocal elliptic equations, ultra-analyticity, Fourier multipliers}

    \vspace{5mm}

\begin{abstract}
It is well known that solutions of elliptic equations inherit analyticity from analytic coefficients,
while much less is understood about the inheritance of ultra-analytic regularity,
especially for nonlocal equations.
This paper develops a systematic Fourier-analytic framework to study fractional 
and more general nonlocal pure-potential equations whose potentials satisfy 
ultra-analytic derivative bounds.
We prove sharp quantitative logarithmic ultra-analytic estimates for normalized solutions,
and show that both the logarithmic power and the leading constant involving the fractional exponent are optimal 
in natural periodic model examples.
We also establish a general transfer principle for weighted ultra-analytic scales,
which reveals why standard scales are not preserved,
and singles out a natural family of invariant ultra-analytic spaces.
\end{abstract}

    \maketitle

     \setcounter{tocdepth}{2}

\section{Introduction and main results}

\subsection{Background and scope}

The analyticity of solutions to elliptic equations with analytic coefficients is a classical theme in elliptic regularity theory; see Morrey--Nirenberg \cite{MorreyNirenberg57} and H\"ormander \cite{Hoermander85}.  Elliptic regularity also extends to Gevrey classes. Under corresponding coefficient assumptions, solutions enjoy the same Gevrey order \cite{Rodino93}. The analytic and Gevrey theories thus provide a classical same‑scale inheritance paradigm, wherein regularity assumptions on the coefficients are transferred to solutions within the same regularity scale.

In the present paper we are concerned with a stronger and much less developed regime, namely ultra-analytic regularity.  A typical quantitative form is Fourier decay faster than exponential,
\[
|\widehat f(\xi)|\lesssim e^{-c|\xi|^\rho},\qquad \rho>1,
\]
which corresponds, up to harmless changes of constants, to estimates of the form
\begin{align}\label{equ-ultra-f-gamma}
	\|D^\alpha f\|_{L^p}
	\le C A^{|\alpha|}(\alpha!)^\kappa,
	\qquad \kappa=\rho^{-1}<1 .
\end{align}
The central issue is whether the same-scale inheritance paradigm continues to hold in this ultra-analytic regime, and if not, what scale should replace it.  This leads to the guiding question of this paper:
\begin{quote}
	If the coefficient \(V\) in an elliptic or nonlocal equation is ultra-analytic, how much ultra-analyticity can a normalized solution inherit?
\end{quote}

We answer this question for pure-potential fractional and nonlocal equations
\[
\Lambda^s u=Vu,\qquad \varphi(D)u=Vu.
\]
Our main finding is that the classical same-scale inheritance principle breaks down in the ultra-analytic regime.  Even when the potential satisfies ultra-analytic derivative bounds, a normalized solution generally inherits only a logarithmically renormalized ultra-analytic scale.  This renormalized scale is sharp in natural periodic model examples and is dictated by the growth of the underlying Fourier multiplier.

This logarithmic ultra-analytic preservation problem was initiated in the second order elliptic setting by Dong and Wang \cite{DongWang24}.  The present paper systematically develops its fractional and nonlocal counterpart by a different Fourier-analytic method.  The contribution of the paper has three parts. 

\begin{itemize}
    \item {\it Fractional $L^p$ theory}. For fractional pure-potential equations, ultra-analytic coefficient bounds imply a sharp logarithmically renormalized solution scale rather than preservation of the original coefficient scale. 
    
\item {\it General multipliers}. For general admissible radial Fourier multipliers, the inherited scale is determined by the growth of the multiplier.  

\item {\it Invariant classes}. For general weighted ultra-analytic classes, the equation acts on the defining scale through a logarithmic transfer law: for the fractional pure-potential equation $\Lambda^s u=Vu$, a coefficient scale $L_m$ is transferred into the solution scale $L_{\lfloor c\log(m+e)\rfloor}$.  This leads to natural invariant ultra-analytic spaces, namely those whose defining weights are stable under the logarithmic renormalization $m\mapsto\log m$.
\end{itemize}

\subsection{Fractional logarithmic ultra-analytic estimates}
We first formulate the coefficient-side ultra-analytic assumption used in the fractional model case.
We assume  
\begin{align}\label{assump-A-derivative}
 \|D^\alpha V\|_{L^\infty(\mathbb R^n)}
 \le A_0M^{|\alpha|}(\alpha!)^\kappa,
 \qquad \alpha\in\mathbb N^n,
\end{align}
where $A_0,M>0$ and $0\le\kappa<1$.
The case $\kappa=0$ includes trigonometric and quasi-periodic potentials, 
such as $V(x)=\sum c_j e^{i\lambda_j\cdot x}$ and Mathieu-type $a+b\cos 2x$, 
which are ubiquitous in the spectral theory of Schr\"odinger operators 
\cite{Bourgain05,KarpeshinaParnovskiShterenberg26}.
For $0<\kappa<1$, \eqref{assump-A-derivative} corresponds to 
faster-than-exponential Fourier decay, and related ultra-analytic classes 
arise in the study of spectral gaps for Hill operators 
\cite{DjakovMityaginEntire2003,Poschel2011HillWeighted,Trubowitz1977}.

The first result reads as follows.

\begin{theorem}[Fractional $L^p$ theory]\label{thm-fractional-Lp}
Let $s>0, 0\leq \kappa<1$, $1\le p\le\infty$, and let $u\in L^p(\mathbb R^n)$ be a distributional solution of
\[\Lambda^s u=Vu,\qquad \|u\|_{L^p(\mathbb R^n)}=1.\]
Assume \eqref{assump-A-derivative}.  Then for every $\theta>1$ there exist constants $C_\theta,K_\theta>0$, depending only on $s,n,\kappa,p$ and $\theta$, such that for every multi-index $\alpha\in\mathbb N^n$,
\begin{align}\label{eq-main-fractional-Lp}
 \|D^\alpha u\|_{L^p(\mathbb R^n)}
 \le
 C_\theta
 \left[
 \theta M\left(\frac{1-\kappa}{s}\right)^{1-\kappa}
 + K_\theta(1+A_0^{1/s})^{n+1}
 \right]^{|\alpha|}
 \frac{|\alpha|!}
 {\big(\log(|\alpha|+e)\big)^{(1-\kappa)|\alpha|}}.
\end{align}
\end{theorem}

The estimate should be read as a precise loss-of-scale statement.  If the potential has ultra-analytic type $M$ and order $\kappa<1$, then normalized solutions satisfy derivative bounds with denominator
\[
\log^{(1-\kappa)|\alpha|}(|\alpha|+e),
\]
and with leading type
\[
M\left(\frac{1-\kappa}{s}\right)^{1-\kappa},
\]
up to the arbitrarily small multiplicative loss represented by $\theta>1$.  Thus the fractional operator does not preserve the original coefficient scale; it transforms it into a logarithmic scale.

This logarithmic renormalization is sharp. Indeed, Proposition \ref{prop-periodic-sharp-leading-coefficient-fractional} constructs, for every $s>0$ and $0\le\kappa<1$, periodic examples of $\Lambda^s u=Vu$ with
\[\|V^{(m)}\|_{L^\infty(\mathbb T)}\le C_0M^m(m!)^\kappa,
\qquad m=0,1,2,\ldots,\]
for which the estimate obtained from \eqref{eq-main-fractional-Lp} fails if the leading coefficient $\left(\frac{1-\kappa}{s}\right)^{1-\kappa} M$ is replaced with $bM$, where $b<\left(\frac{1-\kappa}{s}\right)^{1-\kappa}$.  In the same construction, the denominator $\log^{(1-\kappa)m}(m+e)$ cannot be replaced with $\log^{(1-\kappa+\varepsilon)m}(m+e)$ for any $\varepsilon>0$.  Although the obstruction is constructed on $\mathbb T$, its periodic extension gives a bounded solution on $\mathbb R$ in the endpoint $p=\infty$ case; hence it rules out such uniform improvements of Theorem \ref{thm-fractional-Lp}.  Thus the fractional estimate has the correct logarithmic exponent and the correct leading coefficient in the large-type parameter $M$.

Let us compare Theorem \ref{thm-fractional-Lp} with the local elliptic result of Dong and Wang \cite{DongWang24}.  They studied elliptic equations with lower-order and gradient terms,
\[
\Delta u=W(x)\nabla u+V(x)u,
\]
where $W$ and $V$ are entire functions of exponential type at most $C_0$, their estimate has the form
\[
|D^\alpha u|_{L^p}
\le
\left(\theta C_0+K\left(1+\frac1{\log\theta}\right)\right)^{|\alpha|}
\frac{|\alpha|!}{\log^{|\alpha|}(|\alpha|+e)},
\qquad \theta>1.
\]
Theorem \ref{thm-fractional-Lp} provides a fractional pure-potential counterpart to this logarithmic phenomenon; the precise comparison is discussed in Remark \ref{remark-DW25}.

\subsection{Nonlocal logarithmic ultra-analytic estimates}

Our second theorem treats general radial Fourier multipliers.  Let
\[\varphi(D)f=\mathcal F^{-1}\big(\varphi(|\xi|)\widehat f(\xi)\big),\]
where \(\varphi:[0,\infty)\to[0,\infty)\) is nondecreasing, \(\varphi(r)>0\) for \(r>0\), and \(\varphi(r)\to\infty\).  The fractional case corresponds to $\varphi(r)=r^s$.  For a general multiplier $\varphi$ we write
\[L_\varphi(r):=\log(e+\varphi(r)).\]

\begin{theorem}[General Fourier multipliers]\label{thm-general-phi}
Let $1\le p\le\infty$.  Assume that $\varphi$ is an admissible radial multiplier in the sense of Definition \ref{def-admissible-varphi}.  Let $u\in L^p(\mathbb R^n)$ solve
\[\varphi(D)u=Vu,\qquad \|u\|_{L^p(\mathbb R^n)}=1.\]
Assume \eqref{assump-A-derivative}.  Then for every $\theta>1$ there exist constants $C_{\theta,A_0},K_{\theta,A_0}>0$, depending only on the admissibility constants of $\varphi$, $n,\kappa,p,\theta$ and $A_0$, such that
\begin{align}\label{eq-main-general-phi}
 \|D^\alpha u\|_{L^p(\mathbb R^n)}
 \le
 C_{\theta,A_0}
 \left[
 \theta(1-\kappa)^{1-\kappa} M
 + K_{\theta,A_0}
 \right]^{|\alpha|}
 \frac{|\alpha|!}
 {L_\varphi(|\alpha|+e)^{(1-\kappa)|\alpha|}}.
\end{align}
In particular, if $\varphi(r)=\log(e+r)$, then $L_\varphi(r)\simeq \log(e+\log(e+r))$ and
\begin{align}\label{eq-main-log-multiplier}
 \|D^\alpha u\|_{L^p}
 \le
 C_{\theta,A_0}
 \left[
 \theta(1-\kappa)^{1-\kappa} M
 +K_{\theta,A_0}
 \right]^{|\alpha|}
 \frac{|\alpha|!}
 {\big(\log(e+\log(e+|\alpha|))\big)^{(1-\kappa)|\alpha|}}.
\end{align}
\end{theorem}

For the logarithmic multiplier $\varphi(r)=\log(e+r)$, Proposition \ref{prop-periodic-sharp-loglog-power} shows that the double-logarithmic scale is optimal: for every $0\le\kappa<1$ and every $\varepsilon>0$, there are periodic examples of $\log(e+|D|)u=Vu$ with
\[\|V^{(m)}\|_{L^\infty(\mathbb T)}\le C A^m(m!)^\kappa,
\qquad m=0,1,2,\ldots,\]
for which no estimate can hold with the denominator $\bigl(\log(e+\log(e+m))\bigr)^{(1-\kappa+\varepsilon)m}$ in place of the sharp denominator $\bigl(\log(e+\log(e+m))\bigr)^{(1-\kappa)m}$.  Again, after periodic extension these examples are bounded solutions on $\mathbb R$ and therefore provide endpoint $L^\infty$ obstructions for the corresponding real-line estimate.

\subsection{Invariant ultra-analytic classes}

The third part of the paper addresses a structural question suggested by Theorem~\ref{thm-fractional-Lp}.  The preceding estimates show that the stronger ultra-analytic scales considered here are generally not preserved by the solution map
\[
V\mapsto u,
\qquad \Lambda^s u=Vu.
\]
For a scale $L=\{L_m\}_{m\ge 0}$, derivatives of order $m$ are measured by the denominator $L_m^m$, as in
\[
\|D^\alpha f\|_{L^p}
\lesssim A^{|\alpha|}\frac{|\alpha|!}{L_{|\alpha|}^{|\alpha|}}.
\]
If $V$ has the polynomial ultra-analytic scale $L_m=m^\nu$, then the solution has only the logarithmic scale
\( L_{\lfloor c\log(m+e)\rfloor}\simeq (\log m)^\nu\). Thus the equation typically renormalizes the scale by sending the order $m$ to $\log m$, rather than preserving a fixed ultra-analytic scale.  It is therefore natural to ask which ultra-analytic spaces remain invariant under this logarithmic renormalization.

\begin{quote}
	For which scales $L$ does
	\[
	V\in\mathcal U_L^\infty
	\quad\Longrightarrow\quad
	u\in\mathcal U_L^p
	\]
	hold for normalized solutions of $\Lambda^s u=Vu$?
\end{quote}

Before giving the general criterion, we single out a concrete invariant class.  Let $\log^*m$ be the number of logarithms needed to bring $m$ below a fixed threshold, as defined precisely in Section~\ref{sec-invariant-classes}.  Its basic property is
\[
\log^*(\log m)=\log^*m-1
\]
for all sufficiently large $m$.  Hence powers of $1+\log^*m$ are stable, up to constants, under $m\mapsto\log m$.  Indeed, for
\[
L_m=(1+\log^*m)^\gamma,
\]
the transfer rule gives
\[
L_{\lfloor c\log(m+e)\rfloor}
=(1+\log^*(\lfloor c\log(m+e)\rfloor))^\gamma
\simeq (1+\log^*m)^\gamma=L_m.
\]
The logarithmic renormalization only shifts $\log^*m$ by a bounded amount, which is absorbed by the constants in the ultra-analytic norm.  This explains why iterated-logarithm scales are invariant, whereas polynomial and ordinary logarithmic scales are not.

For $1\le p\le\infty$ and $\gamma>0$, define $\mathcal U_{*,\gamma}^p$ by requiring that there exist $A,C>0$ such that
\[
\|D^\alpha f\|_{L^p(\mathbb R^n)}
\le
C A^{|\alpha|}
\frac{|\alpha|!}{(1+\log^*(|\alpha|+e))^{\gamma |\alpha|}},
\qquad \alpha\in\mathbb N^n .
\]
Equivalently, $\mathcal U_{*,\gamma}^p$ is the total-order ultra-analytic space associated with
\[
L_m=(1+\log^*(m+e))^\gamma .
\]

\begin{theorem}[A concrete invariant ultra-analytic class]
	\label{thm-concrete-invariant-class}
	Let $s>0$, $1\le p\le\infty$, and $\gamma>0$.  Let $u\in L^p(\mathbb R^n)$ be a distributional solution of
	\[
	\Lambda^s u=Vu,
	\qquad
	\|u\|_{L^p(\mathbb R^n)}=1.
	\]
	If $V\in\mathcal U_{*,\gamma}^\infty$, then
	\[
	u\in\mathcal U_{*,\gamma}^p .
	\]
\end{theorem}

This theorem is a concrete instance of the general transfer criterion proved in the final section.  There we introduce total-order classes $\mathcal U_L^p$ associated with admissible weights $L=\{L_m\}$, and show that the fractional equation transfers the coefficient weight $L_m$ to the solution weight
\[
L_{\lfloor c\log(m+e)\rfloor}.
\]
Consequently, $\mathcal U_L^p$ is invariant whenever $L$ is logarithmically stable:
\[
L_{\lfloor c\log(m+e)\rfloor}\gtrsim L_m
\]
for all large $m$.  This condition exactly means that the logarithmic transfer keeps the solution within the original scale.  Standard weights such as $L_m=m^\nu$ ($0<\nu\le1$) and $L_m=(\log(m+e))^\gamma$ fail this stability condition and lose one logarithmic level.  By contrast, the iterated-logarithm weight $L_m=(1+\log^*m)^\gamma$ is stable, and hence yields an invariant ultra-analytic space.

\subsection{Strategy of proof}

We briefly describe the proof mechanisms for the three main quantitative parts of the paper.

First, the fractional estimate is proved by a Fourier-analytic argument rather than by induction on derivatives, which was used in \cite{DongWang24}.  The derivative bounds \eqref{assump-A-derivative} on the potential are converted into refined annular frequency decay: the annular component of \(V\) at frequency radius \(R\) satisfies
\[
\|V_R\|_{L^\infty}
\lesssim A_0\exp\bigl(-c_\kappa(R/M)^{1/\kappa}\bigr),
\qquad 0<\kappa<1,
\]
while for \(\kappa=0\) the potential is of exponential type.  This decay is then combined with an analytic weight \(e^{\sigma|\xi|_\infty}\).  A Legendre transform argument yields a multiplication estimate of the form
\[
\|Vf\|_{\mathcal W_{p,\sigma}}
\lesssim
A_0\exp\bigl(C(1-\kappa)(\sigma M)^{1/(1-\kappa)}+C\sigma\bigr)
\|f\|_{\mathcal W_{p,\sigma}},
\]
where \(\mathcal W_{p,\sigma}\) denotes an appropriate analytic frequency norm.  Using \(\Lambda^s u=Vu\) and splitting \(u\) into low and high frequencies gives
\[
\|u\|_{\mathcal W_{p,\sigma}}
\lesssim e^{C\sigma N}\|u\|_{L^p}
+N^{-s}\|Vu\|_{\mathcal W_{p,\sigma}}.
\]
Choosing
\[
N^s\simeq
A_0\exp\bigl(C(1-\kappa)(\sigma M)^{1/(1-\kappa)}+C\sigma\bigr)
\]
absorbs the high-frequency term.  Finally, extracting \(m\) derivatives from the analytic weight costs \(m!/\sigma^m\); optimizing
\[
\sigma\simeq B_\theta^{-1}\log^{1-\kappa}(m+e)
\]
yields the leading scale \(M((1-\kappa)/s)^{1-\kappa}\) and the denominator \(\log^{(1-\kappa)m}(m+e)\).  For \(p=2\) the argument is implemented directly on the Fourier side, while the remaining \(L^p\) cases follow the same absorption principle using frequency-uniform analytic norms.

Second, for a general admissible radial multiplier \(\varphi(D)\), the same low--high frequency absorption is used with one modification: the absorbing threshold is chosen from
\[
\varphi(N)\sim
\exp\bigl((1-\kappa)(\sigma M)^{1/(1-\kappa)}\bigr)
\]
instead of
\[
N^s\sim
\exp\bigl((1-\kappa)(\sigma M)^{1/(1-\kappa)}\bigr).
\]
Equivalently, the optimization is governed by \(\log(e+\varphi(N))\), which is why the final denominator is
\[
L_\varphi(m)^{(1-\kappa)m},
\qquad L_\varphi(r)=\log(e+\varphi(r)).
\]
Thus \(\varphi(r)=r^s\) gives back the fractional single-logarithmic scale, whereas \(\varphi(r)=\log(e+r)\) gives the double logarithm.  The dependence on \(A_0\) enters only through the absorbing threshold and is absorbed into the constants \(C_{\theta,A_0}\) and \(K_{\theta,A_0}\).

Third, the weighted ultra-analytic transfer theorem is obtained by applying the same quantitative mechanism to general total-order scales.  The fractional equation transforms a coefficient scale \(L_m\) into the solution scale
\[
L_{\lfloor c\log(m+e)\rfloor}.
\]
Thus the proof of invariance reduces to a stability property of the weight sequence: if
\[
L_{\lfloor c\log(m+e)\rfloor}\gtrsim L_m
\]
for all large \(m\), then the loss produced by the equation remains within the original ultra-analytic norm.  The iterated-logarithm weights satisfy this condition because applying one logarithm changes \(\log^*m\) by only a bounded amount.  This shows that the invariant-space result is not an isolated example, but a structural consequence of the same Fourier-analytic transfer mechanism.

\subsection{Connections and potential applications}

We now recall several contexts in which quantitative analytic, Gevrey, and ultra-analytic estimates arise.  Beyond qualitative analyticity, high-order derivative estimates describe the effective analytic radius of solutions and play an important role in spectral geometry, for instance through analytic-continuation and doubling arguments for nodal sets \cite{DonnellyFefferman88}; in heat observability and null controllability through spectral inequalities and propagation of smallness \cite{ApraizEscauriaza14,EscauriazaMontanerZhang15,WangZhang23}; and in fluid models as quantitative measures of smoothing and persistence of analyticity \cite{FoiasTemam89,KukavicaVicol09,PaicuVicol11,BaeBiswasTadmor12,BradshawGrujicKukavica15}.  Related Gevrey estimates for dissipative and fractional-dissipative equations, including critical Besov settings and quasi-geostrophic models, can be found in \cite{BaeBiswas15,Biswas12,BiswasMartinezSilva15,Li24}.  Modern Gelfand--Shilov formulations can be found in \cite{CappielloToft2017,Lopes2017GelfandShilov}.

Ultra-analytic functions also arise naturally in uncertainty principles.  For example, the Hardy uncertainty principle is closely related to Gaussian decay \cite{Hardy33}, and ultra-analytic weights have been used to establish quantitative unique continuation and observability inequalities for dispersive equations; see, for instance, \cite{LiWang21,WangWangZhang19}.  

These connections suggest that the logarithmic ultra-analytic estimates developed here may be useful in problems involving nonlocal elliptic equations, spectral analysis, quantitative propagation phenomena, nonlocal Schr\"odinger operators, quasi-periodic models, pseudo-differential spectral problems, quantitative unique continuation, and observability.

The paper is organized as follows.  Section \ref{sec-fractional-proof} proves Theorem \ref{thm-fractional-Lp}.  Its first subsection collects the refined annular frequency estimates for the potential that are common to both the \(L^2\) and frequency-uniform arguments.  The second subsection gives the direct Fourier-side \(L^2\) proof, including the \(L^2\) multiplication estimate, Fourier decay, and the corresponding derivative bound.  The third subsection introduces the unit-frequency decomposition and proves the remaining \(L^p\) estimates by frequency-uniform Besov-type norms.  Section \ref{sec-general-phi} proves Theorem \ref{thm-general-phi} and discusses the logarithmic multiplier.  Section \ref{sec-sharpness-discussion} contains the sharpness examples.  Section \ref{sec-invariant-classes} develops a transfer criterion for ultra-analytic scales satisfying explicit admissibility hypotheses and proves Theorem \ref{thm-concrete-invariant-class}.

\section{Proof of Theorem \ref{thm-fractional-Lp}}\label{sec-fractional-proof}

This section is devoted to the fractional equation.  We separate the proof into
three parts.  The first subsection contains a refined annular decomposition of the potential and a summation estimate.  These facts are independent of the functional setting and will be used in both the $L^2$ and $L^p$ arguments.  The second subsection then proves the
\(L^2\) case by a short Fourier-side argument based on Plancherel's theorem.  The
third subsection treats the remaining $L^p$ cases; this part requires the
unit-frequency decomposition and a more detailed frequency-uniform Besov-type
analysis.

\subsection{Refined frequency decay of the potential}

The estimates in this subsection depend only on the derivative bounds for the
potential.  We use a refined annular decomposition with ratio $\rho=1+\delta$, rather than a fixed dyadic decomposition. The reason is that multiplication by a coefficient piece whose Fourier support is contained in a ball of radius $R$ costs essentially $e^{\sigma R}$ in the analytic Fourier weight. A dyadic annulus would introduce a fixed enlargement of the radius, and hence a fixed extra factor in this exponential cost. By taking the annular ratio $\rho=1+\delta$ with $\delta>0$ arbitrarily small, the support enlargement becomes only $(1+\delta)R_j$. This allows the coefficient in front of the analytic weight to be made arbitrarily close to one, which is important for preserving the sharp leading scale after the Legendre-transform summation.

Let \(|\xi|_\infty=\max_{1\le j\le n}|\xi_j|\).  Fix
\(0<\delta<1\) and set
\[
        \rho:=1+\delta.
\]
Choose smooth functions \(\chi_{\le0},\chi_j\in C_c^\infty(\mathbb R^n)\),
\(j\ge1\), such that
\[
        1=\chi_{\le0}(\zeta)+\sum_{j\ge1}\chi_j(\zeta),
        \qquad \zeta\in\mathbb R^n,
\]
\[
        \operatorname{supp}\chi_{\le0}
        \subset\{|\zeta|_\infty\le2\},
        \qquad
        \operatorname{supp}\chi_j
        \subset\{\rho^{j-1}\le |\zeta|_\infty\le \rho^{j+1}\}.
\]
For \(M>0\), define
\[
        V_{\le 0}:=\chi_{\le0}(D/M)V,
        \qquad
        V_j:=\chi_j(D/M)V,
        \qquad
        R_j:=M\rho^j .
\]
Then
\[
        V=V_{\le 0}+\sum_{j\ge1}V_j,
\]
with
\[
        \operatorname{supp}\widehat{V_{\le 0}}
        \subset\{|\xi|_\infty\le 2M\},
        \qquad
        \operatorname{supp}\widehat{V_j}
        \subset\{|\xi|_\infty\le \rho R_j\}.
\]

\begin{lemma}[Refined annular frequency decay of \(V\)]\label{lem-refined-V-decay}
Assume \eqref{assump-A-derivative}.  If \(0<\kappa<1\), then for every
\(\delta\in(0,1)\),
\[
        \|V_{\le 0}\|_{L^\infty}\le C_nA_0
\]
and
\[
        \|V_j\|_{L^\infty}
        \le
        C_{\delta,n,\kappa}A_0
        \exp\left(
        -(1-\delta)\kappa
        \left(\frac{R_j}{M}\right)^{1/\kappa}
        \right),
        \qquad j\ge1.
\]
If \(\kappa=0\), then \(V\) is of exponential type \(M\) in the
\(|\cdot|_\infty\) sense, namely
\[
        \operatorname{supp}\widehat V
        \subset
        \{\xi\in\mathbb R^n: |\xi|_\infty\le M\}.
\]
\end{lemma}

\begin{proof}
The estimate for \(V_{\le 0}\) follows from the \(L^1\)-boundedness of the
kernel of \(\chi_{\le0}(D/M)\) and \(\|V\|_\infty\le A_0\).

We prove the estimate for \(V_j\), \(j\ge1\), in the case \(0<\kappa<1\).  On
the support of \(\chi_j(\xi/M)\), the frequency size is comparable to \(R_j\),
with constants depending only on \(\delta\).  We split the annulus into
coordinate pieces.  Choose smooth functions \(\omega_\ell\), \(1\le\ell\le n\),
such that they sum to one on \(\{|\xi|_\infty\ge1/2\}\) and
\[
        \operatorname{supp}\omega_\ell
        \subset
        \{|\xi_\ell|\ge c_{\delta,n}|\xi|_\infty\}.
\]
It suffices to estimate \(\chi_j(D/M)\omega_\ell(D/M)V\).  On the corresponding
support, \(|\xi_\ell|\ge c_{\delta,n}R_j\).  For an integer \(N\ge0\), factor
\[
        \chi_j(\xi/M)\omega_\ell(\xi/M)
        =R_j^{-N}m_{j,\ell,N}(\xi)(i\xi_\ell)^N.
\]
After rescaling the annulus to unit size, the inverse Fourier transform of
\(m_{j,\ell,N}\) has \(L^1\)-norm bounded by
\[
        C_{\delta,n}(1+N)^{C_{\delta,n}}.
\]
Therefore, using \eqref{assump-A-derivative},
\[
        \|V_j\|_\infty
        \le
        C_{\delta,n}A_0
        \inf_{N\ge0}
        (1+N)^{C_{\delta,n}}
        \left(\frac{M}{R_j}\right)^N(N!)^\kappa .
\]
Stirling's formula and the choice
\[
N\simeq \left(\frac{R_j}{M}\right)^{1/\kappa}
\]
give the asserted bound. Indeed, the factor $(1+N)^{C_{\delta,n}}$ is of lower order compared with
$\exp(c(R_j/M)^{1/\kappa})$, and can therefore be absorbed by replacing the coefficient $\kappa$ in the exponent with $(1-\delta)\kappa$.

If \(\kappa=0\), then
\[
        \|D^\alpha V\|_\infty\le A_0M^{|\alpha|}.
\]
Thus \(V\) is a bounded entire function of exponential type at most \(M\) in
each coordinate.  By the Bernstein--Paley--Wiener theorem, \( \operatorname{supp}\widehat V
        \subset
        \{\xi\in\mathbb R^n: |\xi|_\infty\le M\}\).
\end{proof}

We need a technical summation lemma.
\begin{lemma}\label{lem-lacunary-summation}
Let \(0<\kappa<1\), \(M>0\), \(\sigma>0\), and
\(R_j=M(1+\delta)^j\).  Fix \(b_0>0\).  Then for every \(a>0\) and every
\(b\ge b_0\),
\[
        \sum_{j\ge1}
        \exp\left(
        a\sigma R_j
        -b\kappa\left(\frac{R_j}{M}\right)^{1/\kappa}
        \right)
        \le
        C_{\delta,\kappa,b_0}
        \exp\left(
        (1-\kappa) a^{1/(1-\kappa)}\bigl((1-\delta)b\bigr)^{-\kappa/(1-\kappa)}
        (\sigma M)^{1/(1-\kappa)}
        \right).
\]
\end{lemma}

\begin{proof}
Write \(R_j=Mr_j\), where \(r_j=(1+\delta)^j\), and set \(x=\sigma M\).  For
every \(r\ge0\),
\[
        axr-b\kappa r^{1/\kappa}
        =axr-(1-\delta)b\kappa r^{1/\kappa}-\delta b\kappa r^{1/\kappa}.
\]
The supremum for $r\geq 0$ of the first two terms is \((1-\kappa) a^{1/(1-\kappa)}\bigl((1-\delta)b\bigr)^{-\kappa/(1-\kappa)}x^{1/(1-\kappa)}\).
Hence
\[
\begin{aligned}
        \sum_{j\ge1}e^{a x r_j-b\kappa r_j^{1/\kappa}}
        &\le
        \exp\left(
        (1-\kappa) a^{1/(1-\kappa)}\bigl((1-\delta)b\bigr)^{-\kappa/(1-\kappa)}x^{1/(1-\kappa)}
        \right)
        \sum_{j\ge1}e^{-\delta b\kappa r_j^{1/\kappa}} \\
        &\le
        C_{\delta,\kappa,b_0}
        \exp\left(
        (1-\kappa) a^{1/(1-\kappa)}\bigl((1-\delta)b\bigr)^{-\kappa/(1-\kappa)}x^{1/(1-\kappa)}
        \right).
\end{aligned}
\]
The last sum is bounded uniformly for \(b\ge b_0\), since replacing \(b\) by \(b_0\)
gives a convergent majorant whose terms decay like
\(\exp(-c(1+\delta)^{j/\kappa})\).
\end{proof}

\subsection[Fourier-side L2 proof]{Fourier-side \texorpdfstring{$L^2$}{L2} proof}\label{subsec-L2-proof}  
The main advantage when \(p=2\) is that Plancherel's theorem allows the equation to be estimated directly on the Fourier side. We therefore use an exponentially weighted Fourier norm and derive a Fourier-decay estimate first; the desired \(L^2\) derivative bounds then follow by a standard optimization.

For $\sigma>0$ define
\begin{align}\label{def-G-sigma}
 \|f\|_{\mathcal G^\sigma_\infty}
 :=
 \left(\int_{\R^n}e^{2\sigma|\xi|_\infty}|\widehat f(\xi)|^2\, \d\xi\right)^{1/2}.
\end{align}
This is the natural weighted Fourier norm for the Hilbert-space part of the
argument.

\begin{lemma}[Fourier-side localized multiplication]\label{lem-L2-localized-mult}
Let $\sigma>0$, and let $a\in L^\infty(\R^n)$ satisfy
\[
 \operatorname{supp}\widehat a\subset\{\xi:|\xi|_\infty\le R\}.
\]
Then for every $\varepsilon\in(0,1)$ there exists $C_{\varepsilon,n}>0$ such that
\begin{align}\label{eq-localized-mult-G}
 \|af\|_{\mathcal G^\sigma_\infty}
 \le
 C_{\varepsilon,n}
 e^{(1+\varepsilon)\sigma R+C_{\varepsilon,n}\sigma}
 \|a\|_{L^\infty}
 \|f\|_{\mathcal G^\sigma_\infty}.
\end{align}
\end{lemma}

\begin{proof}
Fix $h=\varepsilon(1+R)$ and let
$Q_{h,m}$ denote the  Fourier projection to
\[
 E_{h,m}:=\{\xi\in \R^n: mh\le |\xi|_\infty<(m+1)h\},\qquad m\ge0.
\]
The norm \eqref{def-G-sigma} is equivalent, up to the harmless factor
$e^{C\sigma h}$, to
\[
 \left(\sum_{m\ge0} e^{2\sigma hm}\|Q_{h,m}f\|_2^2\right)^{1/2}.
\]
Since \(\widehat a\) is supported in \(\{|\xi|_\infty\le R\}\), convolution in
frequency shows that \(Q_{h,m}(aQ_{h,\ell}f)\) can be nonzero only when the two
shells satisfy
\[
 |m-\ell|h\le R+Ch,
\]
where the term \(Ch\) accounts for the thickness of the shells.
For such $m$ and $\ell$ one has
$hm\le h\ell+R+Ch$.  Moreover, for each fixed $m$ the number of admissible
indices $\ell$ is bounded by a constant depending only on $\varepsilon$ and $n$.
Using $\|ag\|_2\le\|a\|_\infty\|g\|_2$ and summing in $\ell^2$, we get
\[
 \left(\sum_m e^{2\sigma hm}\|Q_{h,m}(af)\|_2^2\right)^{1/2}
 \le
 C_{\varepsilon,n}e^{\sigma R+C(1+R)\varepsilon\sigma}
 \|a\|_\infty
 \left(\sum_\ell e^{2\sigma h\ell}\|Q_{h,\ell}f\|_2^2\right)^{1/2}.
\]
Returning to the weighted Fourier norm and enlarging the constants gives
\eqref{eq-localized-mult-G}.
\end{proof}

\begin{lemma}[Fourier-side multiplication by \(V\)]\label{lem-V-mult-G}
Assume \eqref{assump-A-derivative}. For every \(\varepsilon\in(0,1)\) there
exists \(C_{\varepsilon,n,\kappa}>0\) such that
\[
        \|Vf\|_{\mathcal G^\sigma_\infty}
        \le
        C_{\varepsilon,n,\kappa}A_0
        \exp\left(
        (1+\varepsilon)(1-\kappa)(\sigma M)^{1/(1-\kappa)}
        +C_{\varepsilon,n,\kappa}\sigma
        \right)
        \|f\|_{\mathcal G^\sigma_\infty} .
\]
If \(\kappa=0\) , the exponential factor is understood as
\[
        (1+\varepsilon)\sigma M+C_{\varepsilon,n}\sigma .
\]
\end{lemma}

\begin{proof}
If \(\kappa=0\), Lemma \ref{lem-refined-V-decay} gives
\[
        \operatorname{supp}\widehat V\subset\{|\xi|_\infty\le M\}.
\]
Applying Lemma \ref{lem-L2-localized-mult} with \(a=V\) and \(R=M\) gives the
claimed endpoint estimate.

Assume \(0<\kappa<1\).  Fix a small \(\delta\in(0,1)\), to be chosen in terms of
\(\varepsilon\), and use the refined decomposition
\[
        V=V_{\le 0}+\sum_{j\ge1}V_j
\]
from Lemma \ref{lem-refined-V-decay}.  The low-frequency part satisfies
\[
        \|V_{\le 0}f\|_{\mathcal G^\sigma_\infty}
        \le
        C_{\delta,n}A_0
        \exp(C_{\delta,n}\sigma+C_{\delta,n}\sigma M)
        \|f\|_{\mathcal G^\sigma_\infty}.
\]
With \(S=(\sigma M)^{1/(1-\kappa)}\), the elementary inequality \(C_{\delta,n}\sigma M\le \delta S+C_{\delta,n}\)
gives
\[
        \|V_{\le 0}f\|_{\mathcal G^\sigma_\infty}
        \le
        C_{\delta,n}A_0e^{C_{\delta,n}\sigma+\delta S}
        \|f\|_{\mathcal G^\sigma_\infty}.
\]

For the high-frequency pieces, Lemmas \ref{lem-L2-localized-mult} and
\ref{lem-refined-V-decay} give, after choosing the loss in
Lemma \ref{lem-L2-localized-mult} and the annular ratio sufficiently small,
\[
        \|V_jf\|_{\mathcal G^\sigma_\infty}
        \le
        C_{\delta,n,\kappa}A_0e^{C_{\delta,n}\sigma}
        \exp\left(
        (1+\delta)\sigma R_j
        -(1-\delta)\kappa\left(\frac{R_j}{M}\right)^{1/\kappa}
        \right)
        \|f\|_{\mathcal G^\sigma_\infty}.
\]
Summing in \(j\) and using Lemma \ref{lem-lacunary-summation} with
\(a=1+\delta\) and \(b=1-\delta\), we get
\begin{align*}
&\sum_{j\ge1}
        \exp\left(
        (1+\delta)\sigma R_j
        -(1-\delta)\kappa\left(\frac{R_j}{M}\right)^{1/\kappa}
        \right)\\
&        \le
        C_{\delta,\kappa}
        \exp\left(
        (1-\kappa)(1+\delta)^{1/(1-\kappa)}
        \bigl((1-\delta)^2\bigr)^{-\kappa/(1-\kappa)}S
        \right).    
\end{align*}
Combining the low- and high-frequency estimates yields
\[
        \|Vf\|_{\mathcal G^\sigma_\infty}
        \le
        C_{\delta,n,\kappa}A_0e^{C_{\delta,n,\kappa}\sigma}
        \exp\left(
        \delta S+
        (1-\kappa)(1+\delta)^{1/(1-\kappa)}
        \bigl((1-\delta)^2\bigr)^{-\kappa/(1-\kappa)}S
        \right)
        \|f\|_{\mathcal G^\sigma_\infty}.
\]
Choose \(\delta=\delta(\varepsilon,\kappa)>0\) sufficiently small so that
\[
        \delta+
        (1-\kappa)(1+\delta)^{1/(1-\kappa)}
        \bigl((1-\delta)^2\bigr)^{-\kappa/(1-\kappa)}
        \le (1+\varepsilon)(1-\kappa) .
\]
This proves the result.
\end{proof}

\begin{proposition}[Fourier-side weighted estimate and derivative bounds]
\label{prop-L2-apriori-G}
Let \(u\in L^2(\R^n)\) solve \(\Lambda^s u=Vu\) and \(\|u\|_2=1\).  For every
\(\theta>1\) there exist constants \(C_\theta,A_\theta,K_\theta>0\), depending
only on \(s,n,\kappa\) and \(\theta\), such that the following estimates hold.

First, for every \(\sigma>0\),
\begin{align}\label{eq-U-L2-apriori-detailed}
        \|u\|_{\mathcal G^\sigma_\infty}
        \le
        C_\theta\exp\big(C_\theta\sigma N_\sigma\big),
\end{align}
where
\begin{align}\label{eq-Nsigma-frac-L2-detailed}
        N_\sigma:=
        A_\theta(1+A_0^{1/s})
        \exp\left(
        \frac{\theta_*(1-\kappa)}{s}(\sigma M)^{1/(1-\kappa)}
        +C_\theta\sigma
        \right),
        \qquad
        \theta_*:=\frac{1+\theta}{2}.
\end{align}
Moreover, with
\begin{align}\label{eq-Btheta-L2-Fourier}
        B_\theta:=
        \theta M\left(\frac{1-\kappa}{s}\right)^{1-\kappa}
        +K_\theta(1+A_0^{1/s}),
\end{align}
we have
\begin{align}\label{eq-L2-derivative-from-Fourier-side}
        \|D^\alpha u\|_{L^2(\R^n)}
        \le
        C_\theta B_\theta^{|\alpha|}
        \frac{|\alpha|!}{\big(\log(|\alpha|+e)\big)^{(1-\kappa)|\alpha|}}
\end{align}
for every \(\alpha\in\N^n\).
\end{proposition}

\begin{proof}
We first justify the a priori use of the weighted Fourier norm.  For \(L\ge1\), define
the truncated weighted norm
\[
    \|f\|_{\mathcal G^{\sigma,L}_\infty}
    :=
    \left(
        \int_{\mathbb R^n}
        e^{2\sigma(|\xi|_\infty\wedge L)}
        |\widehat f(\xi)|^2\,\d\xi
    \right)^{1/2}, \quad |\xi|_\infty\wedge L=\min\{|\xi|_\infty, L\}.
\]
This norm is finite for every \(f\in L^2\).  The proofs of the Fourier-side localized
multiplication estimate (Lemma~\ref{lem-L2-localized-mult}) and the
multiplication estimate by \(V\) (Lemma~\ref{lem-V-mult-G}) remain valid with
\(\mathcal G^\sigma_\infty\) replaced with
\(\mathcal G^{\sigma,L}_\infty\), with constants independent of \(L\).  Indeed, the only
additional point is the elementary inequality
\[
    |\xi+\eta|_\infty\wedge L
    \le
    (|\xi|_\infty\wedge L)+|\eta|_\infty .
\]
Therefore, the absorption argument below is first understood in the truncated norm
\(\mathcal G^{\sigma,L}_\infty\).  To avoid cumbersome notation, we write
\(\mathcal G^\sigma_\infty\), \(U_\sigma\), and \(Y_\sigma\) throughout the argument.
All constants are independent of \(L\), and the full weighted estimate follows at the end
by letting \(L\to\infty\).

We first prove the Fourier-side weighted estimate.  Choose
\(\varepsilon>0\) sufficiently small so that \(1+\varepsilon\le\theta_*\).   Set
\[
        U_\sigma:=\|u\|_{\mathcal G^\sigma_\infty},
        \qquad
        Y_\sigma:=\|\Lambda^s u\|_{\mathcal G^\sigma_\infty}.
\]
By the equation and Lemma \ref{lem-V-mult-G},
\begin{align}\label{eq-YU-L2-detailed}
        Y_\sigma
        \le
        C_\theta A_0
        \exp\left(
        \theta_*(1-\kappa)(\sigma M)^{1/(1-\kappa)}
        +C_\theta\sigma
        \right)
        U_\sigma.
\end{align}
Let \(P_{\le N}\) be a smooth Fourier projection to
\(\{|\xi|_\infty\le2N\}\) and put \(P_{>N}=I-P_{\le N}\).  Since
\(\|u\|_2=1\),
\begin{align}\label{eq-L2-low-frequency-bound}
        \|P_{\le N}u\|_{\mathcal G^\sigma_\infty}
        \le e^{2\sigma N}.
\end{align}
Moreover, on the support of \(P_{>N}\) one has \(|\xi|\gtrsim N\), and hence
\begin{align}\label{eq-L2-high-frequency-bound}
        \|P_{>N}u\|_{\mathcal G^\sigma_\infty}
        \le
        C_sN^{-s}\|\Lambda^s u\|_{\mathcal G^\sigma_\infty}
        =
        C_sN^{-s}Y_\sigma.
\end{align}
Combining \eqref{eq-L2-low-frequency-bound} and
\eqref{eq-L2-high-frequency-bound}, we obtain
\begin{align}\label{eq-L2-low-high-combined}
        U_\sigma
        \le e^{2\sigma N}+C_sN^{-s}Y_\sigma.
\end{align}
Using \eqref{eq-YU-L2-detailed} in \eqref{eq-L2-low-high-combined}, we get
\begin{align}\label{eq-L2-absorption-before-choice}
        U_\sigma
        \le e^{2\sigma N}
        +
        C_\theta A_0N^{-s}
        \exp\left(
        \theta_*(1-\kappa)(\sigma M)^{1/(1-\kappa)}
        +C_\theta\sigma
        \right)
        U_\sigma.
\end{align}
Choose \(N=N_\sigma\) as in \eqref{eq-Nsigma-frac-L2-detailed}, with
\(A_\theta\) sufficiently large.  Then the coefficient of \(U_\sigma\) on the
right-hand side of \eqref{eq-L2-absorption-before-choice} is at most \(1/2\).
Therefore,
\[
        U_\sigma\le 2e^{2\sigma N_\sigma},
\]
which proves \eqref{eq-U-L2-apriori-detailed}, after increasing \(C_\theta\). The estimate above is first obtained for
\(\|u\|_{\mathcal G^{\sigma,L}_\infty}\), uniformly in \(L\).  Letting \(L\to\infty\)
and using monotone convergence proves the stated estimate in
\(\mathcal G^\sigma_\infty\).

We now extract the derivative estimate directly from the full family of
Fourier-side weighted estimates. Let \(m=|\alpha|\ge1\).  By Plancherel and the definition of
\(\mathcal G^\sigma_\infty\),
\[
\begin{aligned}
        \|D^\alpha u\|_2
        &\le
        C\sup_{r\ge0}r^m e^{-\sigma r}
        \|u\|_{\mathcal G^\sigma_\infty}        =C
        \left(\frac{m}{e\sigma}\right)^m
        \|u\|_{\mathcal G^\sigma_\infty}
        \le
        C\frac{m!}{\sigma^m}
        \|u\|_{\mathcal G^\sigma_\infty}.
\end{aligned}
\]
Hence, by \eqref{eq-U-L2-apriori-detailed},
\begin{align}\label{eq-L2-direct-derivative-sigma}
        \|D^\alpha u\|_2
        \le
        C_\theta
        \frac{m!}{\sigma^m}
        \exp\big(C_\theta\sigma N_\sigma\big).
\end{align}

Choose a small \(\eta>0\) such that
\[
        \theta_*\theta^{-1/(1-\kappa)}(1+\eta)^{1/(1-\kappa)}<1.
\]
Set
\[
        \sigma_m:=
        \frac{1+\eta}{B_\theta}
        \big(\log(m+e)\big)^{1-\kappa} .
\]
Since \(B_\theta\ge \theta M\left(\frac{1-\kappa}{s}\right)^{1-\kappa}\), we have
\begin{align}\label{eq-L2-q-choose}
        \frac{\theta_*(1-\kappa)}{s}(\sigma_m M)^{1/(1-\kappa)}
        \le
        \theta_*\theta^{-1/(1-\kappa)}(1+\eta)^{1/(1-\kappa)}\log(m+e).
\end{align}
By choosing \(K_\theta\) sufficiently large in \eqref{eq-Btheta-L2-Fourier}, the
term \(C_\theta\sigma_m\) in \eqref{eq-Nsigma-frac-L2-detailed} can be absorbed
into an arbitrarily small multiple of \(\log(m+e)\).  Hence by \eqref{eq-L2-q-choose} there exists
\(q=q_\theta=\theta_*\theta^{-1/(1-\kappa)}(1+\eta)^{1/(1-\kappa)}\in(0,1)\) such that
\[
        N_{\sigma_m}
        \le
        C_\theta(1+A_0^{1/s})(m+e)^q .
\]
Since \(B_\theta\ge K_\theta(1+A_0^{1/s})\), it follows that
\[
        \sigma_m N_{\sigma_m}
        \le
        C_\theta
        \big(\log(m+e)\big)^{1-\kappa} (m+e)^q
        =
        o(m).
\]
After increasing \(C_\theta\) to handle finitely many small values of \(m\), we
therefore have
\[
        \exp\big(C_\theta\sigma_mN_{\sigma_m}\big)\le C(1+\eta)^m.
\]
Substituting \(\sigma=\sigma_m\) into
\eqref{eq-L2-direct-derivative-sigma}, we obtain
\[
\begin{aligned}
        \|D^\alpha u\|_2
        &\le
        C_\theta
        \frac{m!}{\sigma_m^m}
        (1+\eta)^m  =
        C_\theta
        B_\theta^m
        \frac{m!}{\big(\log(m+e)\big)^{(1-\kappa)m}} .
\end{aligned}
\]
This proves \eqref{eq-L2-derivative-from-Fourier-side} for \(m\ge1\).  The case
\(m=0\) follows from \(\|u\|_2=1\), after enlarging \(C_\theta\).
\end{proof}

\begin{remark}[A special case] The logarithmic scale obtained above is sharp for the bounded pure-potential class considered in this paper. It does not exclude stronger ultra-analyticity in special situations outside this class. For example, if \(u\in L^2(\mathbb R^n)\) is an eigenfunction of the harmonic oscillator \(H=-\Delta+|x|^2\) with eigenvalue \(\lambda\), then \[ \Delta u=(|x|^2-\lambda)u . \] Here \(V(x)=|x|^2-\lambda\) is unbounded and hence does not satisfy our coefficient assumptions. The oscillator eigenfunctions satisfy ultra-analytic bounds corresponding to \eqref{equ-ultra-f-gamma} with \(\kappa=1/2\) and \(p=2\); see \cite{BeauchardJamingPravdaStarov21}. Thus this stronger regularity comes from the special spectral structure of the harmonic oscillator, not from a general inheritance principle for bounded potentials. \end{remark}

\subsection[Frequency-uniform Lp proof]{Frequency-uniform \texorpdfstring{$L^p$}{Lp} proof for \texorpdfstring{$p\ne2$}{}}\label{subsec-Lp-proof}

We now prove the remaining cases \(1\le p\le\infty\), \(p\ne2\). Unlike the \(L^2\) case, the argument cannot rely on Plancherel's theorem. We instead use frequency-uniform analytic norms, together with localized multiplication and low--high frequency estimates adapted to \(L^p\). The frequency scale has to be chosen with some care. A fixed unit-scale norm would introduce polynomial losses in the type parameter \(M\). To keep the sharp fractional scale, we use a scale-adapted norm in the proof; the final derivative extraction remains uniform in this scale.

Let $\psi\in C_c^\infty([-2,2]^n)$ satisfy
\begin{align}\label{eq-Lp-psi-defi}
 \sum_{k\in\Z^n}\psi(\xi-k)=1.  
\end{align} 
For $h\ge1$ set
\begin{align}\label{eq-Lp-Box-defi}
 \Box_{h,k}f=\mathcal F^{-1}\big[\psi(\xi/h-k)\widehat f(\xi)\big],
 \qquad k\in\Z^n.
\end{align}
For $\sigma>0$ define the scale-$h$ frequency-uniform analytic norm
\begin{align}\label{def-Up-sigma}
 \|f\|_{\mathcal U^\sigma_{p,h}}
 :=
 \sup_{k\in\Z^n}
 e^{\sigma h|k|_\infty}
 \|\Box_{h,k} f\|_{L^p}.
\end{align}
The scale \(h\) will be chosen comparable to \(1+M\), with the proportionality
factor depending only on the auxiliary parameter \(\varepsilon\). This choice removes the unwanted polynomial dependence
on $M$ while preserving the sharp logarithmic coefficient.

\begin{lemma}[Frequency localized multiplication]\label{lem-localized-mult}
Let $1\le p\le\infty$, $\sigma>0$, and let $a\in L^\infty(\R^n)$ satisfy
\[
 \operatorname{supp}\widehat a\subset\{\xi:|\xi|_\infty\le R\}.
\]
Then, for every $h\ge1$,
\begin{align}\label{eq-localized-mult-h}
 \|af\|_{\mathcal U^\sigma_{p,h}}
 \le
 C_n\left(1+\frac{R}{h}\right)^n
 e^{\sigma R+C_n\sigma h}
 \|a\|_{L^\infty}
 \|f\|_{\mathcal U^\sigma_{p,h}}.
\end{align}
\end{lemma}

\begin{proof}
The Fourier support of $a\Box_{h,k}f$ is contained in the $R$-neighborhood of
that of $\Box_{h,k}f$.  Therefore,
\[
 \Box_{h,\ell}(a\Box_{h,k}f)=0
\]
unless
\[
 h|\ell-k|_\infty\le R+C_nh.
\]
Let $L=C_n+R/h$.  Using the uniform $L^p$ boundedness of the projections
$\Box_{h,\ell}$ and the estimate $\|ag\|_p\le\|a\|_\infty\|g\|_p$, we obtain
\[
 \|\Box_{h,\ell}(af)\|_p
 \le C_n\|a\|_\infty
 \sum_{|\ell-k|_\infty\le L}\|\Box_{h,k}f\|_p.
\]
For such $k$ and $\ell$,
\[
 h|\ell|_\infty\le h|k|_\infty+R+C_nh.
\]
Thus
\[
 e^{\sigma h|\ell|_\infty}\|\Box_{h,\ell}(af)\|_p
 \le C_ne^{\sigma R+C_n\sigma h}\|a\|_\infty
 \sum_{|\ell-k|_\infty\le L}
 e^{\sigma h|k|_\infty}\|\Box_{h,k}f\|_p.
\]
Taking the supremum in $\ell$ gives \eqref{eq-localized-mult-h}, since there are
at most $C_n(1+L)^n$ indices $k$ in the sum.
\end{proof}

\begin{lemma}[Scale-adapted multiplication by \(V\)]\label{lem-V-mult-weighted}
Let \(1\le p\le\infty\) and assume that \(V\) satisfies
\eqref{assump-A-derivative}. For every \(\varepsilon\in(0,1)\) there exist
\(\delta=\delta(\varepsilon,n,\kappa)\in(0,1)\) and
\(C_{\varepsilon,n,\kappa}>0\) such that, with
\[
        h_M:=1+\delta M,
\]
multiplication by \(V\) satisfies
\begin{align}\label{eq-V-mult-symbolic}
        \|Vf\|_{\mathcal U^\sigma_{p,h_M}}
        \le
        C_{\varepsilon,n,\kappa}A_0
        \exp\left((1+\varepsilon)(1-\kappa)(\sigma M)^{1/(1-\kappa)}
        +C_{\varepsilon,n,\kappa}\sigma\right)
        \|f\|_{\mathcal U^\sigma_{p,h_M}}.
\end{align}
In the case \(\kappa=0\), the factor in the exponential is understood as
\[
        (1+\varepsilon)\sigma M+C_{\varepsilon,n}\sigma.
\]
\end{lemma}

\begin{proof}
We first treat the endpoint \(\kappa=0\).  By Lemma \ref{lem-refined-V-decay},
\[
        \operatorname{supp}\widehat V\subset\{|\xi|_\infty\le M\}.
\]
Applying \eqref{eq-localized-mult-h} with \(a=V\), \(R=M\), and
\(h=h_M=1+\delta M\), we get
\[
        \|Vf\|_{\mathcal U^\sigma_{p,h_M}}
        \le
        C_{\delta,n}A_0
        e^{(1+C_n\delta)\sigma M+C_n\sigma}
        \|f\|_{\mathcal U^\sigma_{p,h_M}}.
\]
Choosing \(\delta\) sufficiently small in terms of \(\varepsilon\) gives the
endpoint estimate.

Assume now that \(0<\kappa<1\).  Use the refined decomposition
\[
        V=V_{\le 0}+\sum_{j\ge1}V_j
\]
from Lemma \ref{lem-refined-V-decay}, with the same parameter \(\delta\).  The
low-frequency part satisfies
\[
        \|V_{\le 0}\|_\infty\le C_nA_0,
        \qquad
        \operatorname{supp}\widehat{V_{\le 0}}
        \subset\{|\xi|_\infty\le 2M\}.
\]
Thus \eqref{eq-localized-mult-h} gives
\[
        \|V_{\le 0}f\|_{\mathcal U^\sigma_{p,h_M}}
        \le
        C_{\delta,n}A_0
        \exp(C_{\delta,n}\sigma+C_{\delta,n}\sigma M)
        \|f\|_{\mathcal U^\sigma_{p,h_M}}.
\]
With \(S=(\sigma M)^{1/(1-\kappa)}\), the inequality
\(C_{\delta,n}\sigma M\le \delta S+C_{\delta,n}\) gives
\begin{align}\label{eq-V-low-mult-Up-final}
        \|V_{\le 0}f\|_{\mathcal U^\sigma_{p,h_M}}
        \le
        C_{\delta,n}A_0e^{C_{\delta,n}\sigma+\delta S}
        \|f\|_{\mathcal U^\sigma_{p,h_M}}.
\end{align}

For the high-frequency pieces, Lemma \ref{lem-refined-V-decay} gives
\[
        \|V_j\|_\infty
        \le
        C_{\delta,n,\kappa}A_0
        \exp\left(
        -(1-\delta)\kappa\left(\frac{R_j}{M}\right)^{1/\kappa}
        \right),
\]
and
\[
        \operatorname{supp}\widehat{V_j}
        \subset\{|\xi|_\infty\le (1+\delta)R_j\}.
\]
Applying \eqref{eq-localized-mult-h} to \(V_j\) gives
\[
\begin{aligned}
        \|V_jf\|_{\mathcal U^\sigma_{p,h_M}}
        &\le
        C_{\delta,n,\kappa}A_0
        \left(1+\frac{R_j}{h_M}\right)^{C_n}
        \exp\left((1+\delta)\sigma R_j+C_n\sigma h_M\right) \\
        &\quad\times
        \exp\left(
        -(1-\delta)\kappa\left(\frac{R_j}{M}\right)^{1/\kappa}
        \right)
        \|f\|_{\mathcal U^\sigma_{p,h_M}}.
\end{aligned}
\]
The overlap factor \(\bigl(1+R_j/h_M\bigr)^{C_n}\) is lower order compared with
\[
\exp\left(c\left(\frac{R_j}{M}\right)^{1/\kappa}\right),
\]
and can therefore be absorbed into the exponential decay term, at the cost of replacing \((1-\delta)\kappa\) by \((1-2\delta)\kappa\) in the negative exponent.  Hence
\[
        \|V_jf\|_{\mathcal U^\sigma_{p,h_M}}
        \le
        C_{\delta,n,\kappa}A_0e^{C_{\delta,n,\kappa}\sigma+\delta S}
        \exp\left(
        (1+\delta)\sigma R_j
        -(1-2\delta)\kappa\left(\frac{R_j}{M}\right)^{1/\kappa}
        \right)
        \|f\|_{\mathcal U^\sigma_{p,h_M}}.
\]
Summing over \(j\ge1\) and using Lemma \ref{lem-lacunary-summation} with
\(a=1+\delta\) and \(b=1-2\delta\), we obtain
\begin{align*}
        &\sum_{j\ge1}
        \exp\left(
        (1+\delta)\sigma R_j
        -(1-2\delta)\kappa\left(\frac{R_j}{M}\right)^{1/\kappa}
        \right)\\
&        \le
        C_{\delta,\kappa}
        \exp\left(
        (1-\kappa)(1+\delta)^{1/(1-\kappa)}
        \bigl((1-\delta)(1-2\delta)\bigr)^{-\kappa/(1-\kappa)}S
        \right).
\end{align*}
Combining this with \eqref{eq-V-low-mult-Up-final} gives
\begin{align*}
        &\|Vf\|_{\mathcal U^\sigma_{p,h_M}}\\
        &\le        C_{\delta,n,\kappa}A_0e^{C_{\delta,n,\kappa}\sigma}
        \exp\left(
        \delta S+
        (1-\kappa)(1+\delta)^{1/(1-\kappa)}
        \bigl((1-\delta)(1-2\delta)\bigr)^{-\kappa/(1-\kappa)}S
        \right)
        \|f\|_{\mathcal U^\sigma_{p,h_M}}.
\end{align*}
Finally choose \(\delta=\delta(\varepsilon,\kappa)>0\) sufficiently small so
that
\[
        \delta+
        (1-\kappa)(1+\delta)^{1/(1-\kappa)}
        \bigl((1-\delta)(1-2\delta)\bigr)^{-\kappa/(1-\kappa)}
        \le (1+\varepsilon)(1-\kappa) .
\]
This proves \eqref{eq-V-mult-symbolic}.
\end{proof}

\begin{lemma}[Low and high frequencies in $\mathcal U^\sigma_{p,h}$]
\label{lem-Up-low-high}
Let $h\ge1$.  For $N\ge h$ and $1\le p\le\infty$,
\begin{align}\label{eq-Up-low}
 \|P_{\le N}f\|_{\mathcal U^\sigma_{p,h}}
 &\le C_n e^{C\sigma N}\|f\|_{L^p},\\
\label{eq-Up-high}
 \|P_{>N}f\|_{\mathcal U^\sigma_{p,h}}
 &\le C_{s,n}N^{-s}
 \|\Lambda^s f\|_{\mathcal U^\sigma_{p,h}}.
\end{align}
\end{lemma}

\begin{proof}
For the low-frequency estimate, $\Box_{h,k}P_{\le N}f=0$ unless
$h|k|_\infty\le C_nN$.  Since $\Box_{h,k}$ and $P_{\le N}$ have kernels uniformly
bounded in $L^1$,
\[
 \|\Box_{h,k}P_{\le N}f\|_p\le C_n\|f\|_p.
\]
Thus
\[
 e^{\sigma h|k|_\infty}\|\Box_{h,k}P_{\le N}f\|_p
 \le C_ne^{C\sigma N}\|f\|_p.
\]
Taking the supremum in $k$ proves \eqref{eq-Up-low}.

For the high-frequency estimate, since \(P_{>N}\), \(\Lambda^{-s}\), and \(\Box_{h,k}\) are Fourier multipliers,
they commute.  Hence
\[
 \Box_{h,k}P_{>N}f
 =
 T_{N,k}\Box_{h,k}\Lambda^s f,
\]
where \(T_{N,k}\) is the Fourier multiplier with symbol \((1-\chi(\xi/N))|\xi|^{-s}\)
localized to the support of \(\psi(\xi/h-k)\).  On this support, with \(N\ge h\),
the normalized symbol
\[
 N^s(1-\chi(\xi/N))|\xi|^{-s}
\]
has inverse Fourier kernel bounded uniformly in \(L^1\), after inserting a
cutoff equal to one on \(\supp\psi(\xi/h-k)\).  Therefore,
\[
 \|\Box_{h,k}P_{>N}f\|_p
 \le
 C N^{-s}\|\Box_{h,k}\Lambda^s f\|_p.
\]
Multiplying by \(e^{\sigma h|k|_\infty}\) and taking the supremum gives
\eqref{eq-Up-high}.
\end{proof}

\begin{lemma}[A priori estimate in \(\mathcal U^\sigma_{p,h_M}\)]
\label{lem-Up-apriori}
Let \(u\) be as in Theorem \ref{thm-fractional-Lp}, with \(p\ne2\).
Fix \(\varepsilon\in(0,1)\), and let \(h_M\) be the scale supplied by
Lemma \ref{lem-V-mult-weighted}.  Define
\[
        N_{\sigma,\varepsilon}
        :=h_M+
        A_{\varepsilon,p,s,n,\kappa}
        (1+A_0^{1/s})
        \exp\left(
        \frac{(1+\varepsilon)(1-\kappa)}{s}(\sigma M)^{1/(1-\kappa)}
        +
        C_{\varepsilon,p,s,n,\kappa}\sigma
        \right).
\]
Then, for every \(\sigma>0\),
\begin{align}\label{eq-Up-apriori}
 \|u\|_{\mathcal U^\sigma_{p,h_M}}
 \le
 C_{\varepsilon,p,s,n,\kappa}
 \exp\bigl(
 C_{\varepsilon,p,s,n,\kappa}\sigma N_{\sigma,\varepsilon}
 \bigr).
\end{align}
\end{lemma}

\begin{proof}
We first justify the a priori use of the analytic frequency-uniform norm.  For
\(L\ge1\), define
\[
    \|f\|_{\mathcal U^{\sigma,L}_{p,h}}
    :=
    \sup_{k\in\mathbb Z^n}
    e^{\sigma h(|k|_\infty\wedge L)}
    \|\Box_{h,k}f\|_{L^p}.
\]
This norm is finite for every \(f\in L^p\), since the weight is bounded and the
operators \(\Box_{h,k}\) are uniformly bounded on \(L^p\).  The proofs of
Lemmas~\ref{lem-localized-mult}, \ref{lem-V-mult-weighted}, and
\ref{lem-Up-low-high} remain valid with \(\mathcal U^\sigma_{p,h}\) replaced with
\(\mathcal U^{\sigma,L}_{p,h}\), with constants independent of \(L\).  The only
additional point needed in the block summations is
\[
    |\ell|_\infty\wedge L
    \le
    (|k|_\infty\wedge L)+|\ell-k|_\infty,
    \qquad k,\ell\in\mathbb Z^n .
\]

Set
\[
    U_{\sigma,L}:=\|u\|_{\mathcal U^{\sigma,L}_{p,h_M}},
    \qquad
    Y_{\sigma,L}:=\|\Lambda^s u\|_{\mathcal U^{\sigma,L}_{p,h_M}}.
\]
By Lemma~\ref{lem-V-mult-weighted}, applied in the truncated norm,
\begin{align}\label{eq-Up-YU-mult-truncated}
    Y_{\sigma,L}
    \le
    C_\varepsilon A_0
    \exp\left((1+\varepsilon)(1-\kappa)(\sigma M)^{1/(1-\kappa)}
    +C_\varepsilon\sigma\right)
    U_{\sigma,L}.
\end{align}
Using Lemma~\ref{lem-Up-low-high} in the truncated norm and \(\|u\|_p=1\), we get
\begin{align}\label{eq-Up-low-high-apriori-truncated}
    U_{\sigma,L}
    \le
    C e^{C\sigma N}+CN^{-s}Y_{\sigma,L}.
\end{align}
Combining \eqref{eq-Up-YU-mult-truncated} and
\eqref{eq-Up-low-high-apriori-truncated}, we obtain
\[
    U_{\sigma,L}
    \le
    C e^{C\sigma N}
    +
    C_\varepsilon A_0N^{-s}
    \exp\left((1+\varepsilon)(1-\kappa)(\sigma M)^{1/(1-\kappa)}
    +C_\varepsilon\sigma\right)
    U_{\sigma,L}.
\]
Choose \(N=N_{\sigma,\varepsilon}\), with
\(A_{\varepsilon,p,s,n,\kappa}\) sufficiently large.  Then the coefficient of
\(U_{\sigma,L}\) on the right-hand side is at most \(1/2\).  Therefore,
\[
    U_{\sigma,L}
    \le
    C_{\varepsilon,p,s,n,\kappa}
    \exp\bigl(C_{\varepsilon,p,s,n,\kappa}\sigma N_{\sigma,\varepsilon}\bigr),
\]
with constants independent of \(L\).

Finally, letting \(L\to\infty\), we have \( \|u\|_{\mathcal U^{\sigma,L}_{p,h_M}}
    \uparrow
    \|u\|_{\mathcal U^\sigma_{p,h_M}}.\) 
Hence
\[
    \|u\|_{\mathcal U^\sigma_{p,h_M}}
    \le
    C_{\varepsilon,p,s,n,\kappa}
    \exp\bigl(
    C_{\varepsilon,p,s,n,\kappa}\sigma N_{\sigma,\varepsilon}
    \bigr),
\]
which proves \eqref{eq-Up-apriori}.
\end{proof}

We now finish the proof for \(p\ne2\).  Let \(h=h_M=1+\delta M\geq 1\) be the scale chosen in
Lemma \ref{lem-V-mult-weighted}, and let \(m=|\alpha|\ge1\).  We first extract
derivatives directly from the frequency-uniform norm.  Since
\(\Box_{h,k}u\) is Fourier supported in a cube of size comparable to
\(h(1+|k|_\infty)\), Bernstein's inequality gives
\[
        \|D^\alpha \Box_{h,k}u\|_p
        \le
        \bigl(h(|k|_\infty+C_n)\bigr)^m
        \|\Box_{h,k}u\|_p .
\]
Using the definition of \(\mathcal U^\sigma_{p,h}\), we obtain
\begin{equation}
    \label{eq7.31}
        \|D^\alpha u\|_p
        \le
        \sum_{k\in\Z^n}
        \bigl(h(|k|_\infty+C_n)\bigr)^m
        e^{-\sigma h|k|_\infty}
        \|u\|_{\mathcal U^\sigma_{p,h}}.
\end{equation}
We estimate the lattice sum by comparing it with an integral.   If
\(x\in k+[0,1)^n\), then
\[
        \big||x|_\infty-|k|_\infty\big|\le C_n .
\]
After enlarging \(C_n\) if necessary, this implies
\[
        \bigl(h(|k|_\infty+C_n)\bigr)^m e^{-\sigma h|k|_\infty}
        \le
        e^{C_n\sigma h}
        \bigl(h(|x|_\infty+C_n)\bigr)^m
        e^{-\sigma h|x|_\infty}.
\] 
Therefore,
\begin{equation}
    \label{eq7.32}
\begin{aligned}
        \sum_{k\in\mathbb Z^n}
        \bigl(h(|k|_\infty+C_n)\bigr)^m
        e^{-\sigma h|k|_\infty}
        &\le
        C_n e^{C_n\sigma h}
        \int_{\mathbb R^n}
        \bigl(h(|x|_\infty+C_n)\bigr)^m
        e^{-\sigma h|x|_\infty}\,\d x  \\
        &\le
        C_n e^{C_n\sigma h}
        \int_0^\infty
        \bigl(h(r+C_n)\bigr)^m
        e^{-\sigma h r}
        r^{n-1}\,\d r .
\end{aligned}
\end{equation}
Changing variables \(r+C_n\mapsto r\) and then \(hr\mapsto r\), and using
\(r^{n-1}\le (r+C_n)^{n-1}\), we obtain
\begin{equation}
    \label{eq7.33}
\begin{aligned}
        \int_0^\infty
        \bigl(h(r+C_n)\bigr)^m
        e^{-\sigma h r}
        r^{n-1}\,\d r
        &\le
        C_n e^{C_n\sigma h}h^{-n}
        \int_0^\infty r^{m+n-1}e^{-\sigma r}\,\d r  \\
        &=
        C_n e^{C_n\sigma h}h^{-n}
        \frac{(m+n-1)!}{\sigma^{m+n}}  \\
        &\le
        C_n e^{C_n\sigma h}
        (m+1)^{n-1}
        \frac{m!}{\sigma^m}
        \frac{1}{(\sigma h)^n}. 
\end{aligned}
\end{equation}
Thus, using \eqref{eq7.31}, \eqref{eq7.32}, \eqref{eq7.33}, and Lemma \ref{lem-Up-apriori},
\[
    \|D^\alpha u\|_p
        \le
        C_{\varepsilon,p,s,n,\kappa}
        e^{C_n\sigma h}
        (m+1)^{n-1}
        \frac{m!}{\sigma^m}
        \frac{1}{(\sigma h)^n}
        \exp\bigl(
        C_{\varepsilon,p,s,n,\kappa}\sigma N_{\sigma,\varepsilon}
        \bigr).
\]

We use the same slightly enlarged optimizing scale as in the \(L^2\) argument.  Fix \(\theta>1\).  Choose \(\eta>0\) and
\(\varepsilon=\varepsilon(\theta,\eta)>0\) sufficiently small, and then choose
\(K_\theta\) sufficiently large, so that after setting
\[
        B_\theta=
        \theta M\left(\frac{1-\kappa}{s}\right)^{1-\kappa}
        +K_\theta(1+A_0^{1/s}),
\]
the quantity \(N_{\sigma_m,\varepsilon}\) defined in Lemma
\ref{lem-Up-apriori}, with
\[
        \sigma_m:=
        \frac{1+\eta}{B_\theta}
        \big(\log(m+e)\big)^{1-\kappa} ,
\]
satisfies, for some \(q=q_\theta\in(0,1)\),
\[
        N_{\sigma_m,\varepsilon}
        \le h+
        C_\theta(1+A_0^{1/s})(m+e)^q .
\]
Indeed, using \(B_\theta\ge \theta M\left(\frac{1-\kappa}{s}\right)^{1-\kappa}\), we have
\[
        \frac{(1+\varepsilon)(1-\kappa)}{s}
        (\sigma_m M)^{1/(1-\kappa)}
        \le
        (1+\varepsilon)(1+\eta)^{1/(1-\kappa)}
        \theta^{-1/(1-\kappa)}
        \log(m+e).
\]
Thus \(\eta\) and \(\varepsilon\) are chosen so that
\[
        (1+\varepsilon)(1+\eta)^{1/(1-\kappa)}\theta^{-1/(1-\kappa)}<1.
\]
The remaining term \(C_{\varepsilon}\sigma_m\) in the definition of
\(N_{\sigma_m,\varepsilon}\) is lower order in \(\log(m+e)\), or is absorbed by
taking \(K_\theta\) sufficiently large in the endpoint case.  Hence the above
bound holds for some \(q=q_\theta<1\).

Since \(B_\theta\ge K_\theta(1+A_0^{1/s})\), it follows that
\[
\begin{aligned}
        \sigma_mN_{\sigma_m,\varepsilon}
        &\le \sigma_m h+
        C_\theta
        \frac{1+A_0^{1/s}}{B_\theta}
        \big(\log(m+e)\big)^{1-\kappa}
        (m+e)^q  \\
        &\le C_\theta(\log(m+e))^{1-\kappa}+
        C_\theta
        \big(\log(m+e)\big)^{1-\kappa}
        (m+e)^q
        =
        o(m).
\end{aligned}
\]
Moreover,   the factor \(e^{C\sigma_m h}\) and the polynomial factor
\((m+1)^{n-1}\)  are $e^{o(m)}$.  Hence, after increasing the constant to absorb finitely many small values of
\(m\),
\[
        e^{C_{n}\sigma_m h}
        (m+1)^{n-1}        \exp(C_{\varepsilon,p,s,n,\kappa}\sigma_mN_{\sigma_m,\varepsilon})
        \le
        C_\theta(1+\eta)^m .
\]
Thus, we obtain
\[
\begin{aligned}
        \|D^\alpha u\|_p
         \le
        C_\theta \left(
\frac{B_\theta}
{h\bigl(\log(m+e)\bigr)^{1-\kappa}}
\right)^n
        B_\theta^m
        \frac{m!}{\big(\log(m+e)\big)^{(1-\kappa)m}} .
\end{aligned}
\]

It remains to absorb the factor $(
\frac{B_\theta}
{h\bigl(\log(m+e)\bigr)^{1-\kappa}}
)^n$ into the type constant.  Put
\(Y:=1+A_0^{1/s}\).  Define
\[
B_\theta':= \theta M\left(\frac{1-\kappa}{s}\right)^{1-\kappa} + K_\theta' Y^{n+1},
\]
where \(K_\theta'\) will be chosen sufficiently large. Since \(h=1+\delta M\), we have
\(\frac{M}{h}\le \delta^{-1}\) and \(h^{-1}\le 1\).  Hence
\[
\begin{aligned}
B_\theta \left(\frac{B_\theta}{h}\right)^n
&\le C_\theta (M+Y) \left(\frac{M+Y}{h}\right)^n \\
&\le C_\theta (M+Y) \sum_{j=0}^n \left(\frac{M}{h}\right)^{n-j} \left(\frac{Y}{h}\right)^j \\
&\le C_\theta \bigl(M+Y^{n+1}\bigr) \le C_\theta B_\theta' .
\end{aligned}
\]
Combining this with the previous estimate yields
\[
\|D^\alpha u\|_p \le C_\theta (B_\theta')^m \frac{m!}{\bigl(\log(m+e)\bigr)^{(1-\kappa)m}} .
\]
The case \(m=0\) follows from \(\|u\|_p=1\), after enlarging \(C_\theta\).
This proves Theorem \ref{thm-fractional-Lp} for all
\(1\le p\le\infty\) with \(p\ne2\).  Together with the Fourier-side \(L^2\)
result proved in Subsection \ref{subsec-L2-proof}, the proof of
Theorem \ref{thm-fractional-Lp} is complete.

\begin{remark}\label{remark-DW25}
The strategy developed in this paper also applies to fractional equations with gradient terms when $s>1$. 
More precisely, for the equation $\Lambda^s u = W\cdot\nabla u + V u$, 
under ultra-analytic assumptions on $W$ and $V$ analogous to \eqref{assump-A-derivative}, 
one can obtain logarithmic ultra-analytic estimates of the same form as in Theorem~\ref{thm-fractional-Lp}, 
provided that the smoothing exponent $s$ is replaced with $s-1$. 
This shift reflects the fact that the gradient term consumes one order of smoothing, 
leaving $s-1$ as the effective gain in the high-frequency absorption argument. In particular, when $s=2$ and $\kappa=0$, this recovers the sharp estimate established in \cite{DongWang24}.
\end{remark}

\section{Proof of Theorem \ref{thm-general-phi}}\label{sec-general-phi}

We now replace \(\Lambda^s\) with a general radial Fourier multiplier \(\varphi(D)\).  The frequency decay of the potential and the scale-adapted multiplication estimate remain unchanged.  The only new point is the high-frequency estimate: instead of gaining a factor \(N^{-s}\), one gains the inverse of \(\varphi(N)\).

\begin{definition}[Admissible radial multipliers]\label{def:admissible-multiplier}\label{def-admissible-varphi} Let \(\varphi:[0,\infty)\to[0,\infty)\) be nondecreasing, satisfy \[ \varphi(r)>0\quad\text{for }r>0, \qquad \lim_{r\to\infty}\varphi(r)=\infty . \] We denote by \[ \Phi_\varphi(t):=\inf\{R\ge0:\varphi(R)\ge t\},\qquad t>0, \] its generalized inverse. We say that \(\varphi\) is an admissible radial multiplier if the following two conditions hold. \smallskip 

\noindent \((i)\) For every \(a\in(0,1)\) and every \(A\ge1\), there exist \(q\in(0,1)\) and \(C=C(a,A,\varphi)>0\) such that 
\begin{align}\label{assump-admissible-varphi-1}
\Phi_\varphi\bigl(A(e+\varphi(r))^a\bigr) \le C(e+r)^q,\qquad r\ge0 .     
\end{align}
   
   \smallskip 
   \noindent \((ii)\) There exists an integer \(N_*>n+1\) such that \(\varphi\in C^{N_*}((0,\infty))\) and
\begin{align}\label{assump-admissible-varphi-2}
\sup_{r>0} \left| \frac{(r\partial_r)^j\varphi(r)}{\varphi(r)} \right| <\infty, \qquad 1\le j\le N_* .    
\end{align}   
The constants in \eqref{assump-admissible-varphi-1} and \eqref{assump-admissible-varphi-2} will be called the admissibility constants of \(\varphi\). \end{definition}

\begin{remark}\label{rem:admissible-symbol-condition}
The assumption $(i)$ implies that, for every
\(\rho>0\), \( L_\varphi(m)=o(m^\rho)\).
 To see this, fix \(a\in(0,1)\) in \eqref{assump-admissible-varphi-1}.  The inverse
estimate gives, after taking logarithms,
\[
        L_\varphi(r)
        \le a^{-1}L_\varphi\big(C(e+r)^q\big)+C_a
\]
for some \(q\in(0,1)\).  Iterating this inequality reduces \(r\) to a bounded
range after \(O(\log\log r)\) steps, and yields
\[
        L_\varphi(r)\le C(\log(e+r))^\mu
\]
for some \(\mu>0\).  This is \(o(r^\rho)\) for every \(\rho>0\).
\end{remark}

Let \(\varphi\) be admissible in the sense of Definition \ref{def-admissible-varphi}.  Throughout
this section, we fix a cutoff \(\chi\in C_c^\infty(\mathbb R^n)\) satisfying
\[
        0\le \chi\le1,\qquad
        \chi(\xi)=1\ \text{for } |\xi|\le1,
        \qquad
        \chi(\xi)=0\ \text{for } |\xi|\ge2,
\]
and set
\[
        P_{\le N}:=\chi(D/N),\qquad P_{>N}:=I-P_{\le N}.
\]
We use the frequency-uniform cutoff \(\psi\) and the operators \(\Box_{h,k}\)
defined in \eqref{eq-Lp-psi-defi} and \eqref{eq-Lp-Box-defi}.

\begin{lemma}[High-frequency estimate for \(\varphi(D)\)]\label{lem-phi-high}
Let \(1\le p\le\infty\), \(h\ge1\), and \(N\ge h\).  Then
\begin{align}\label{eq-phi-high}
        \|P_{>N}f\|_{\mathcal U^\sigma_{p,h}}
        \le C_\varphi\varphi(N)^{-1}\|\varphi(D)f\|_{\mathcal U^\sigma_{p,h}}.
\end{align}
Here \(C_\varphi>0\) is independent of \(N,h,\sigma\) and \(f\).
\end{lemma}

\begin{proof}
Write $P_{>N}=T_N\varphi(D)$, where
\[
T_N=\mathcal F^{-1}\big[(1-\chi(\cdot/N))\varphi(|\cdot|)^{-1}\big]
=\varphi(N)^{-1}\widetilde T_N,
\]
and $\widetilde T_N$ has symbol $m_N(\xi)=\varphi(N)(1-\chi(\xi/N))\varphi(|\xi|)^{-1}$.
It suffices to show $\|\widetilde T_N g\|_{\mathcal U^\sigma_{p,h}}\le C_\varphi\|g\|_{\mathcal U^\sigma_{p,h}}$ uniformly for $N\ge h$.

The admissibility condition on $\varphi$ yields symbol estimates
\[
|\partial_\xi^\gamma m_N(\xi)|\le C_\gamma|\xi|^{-|\gamma|},\quad |\gamma|\le N_*,
\]
on the support of $1-\chi(\xi/N)$. Pick $\widetilde\psi=1$ near $\operatorname{supp}\psi$. For each $k$, define $a_{N,h,k}(\eta)=\widetilde\psi(\eta-k)m_N(h\eta)$. Since $|h\eta|\ge N$ on the support of relevant derivatives, we have
\[
\sup_{\eta}|\partial_\eta^\alpha a_{N,h,k}(\eta)|\le C_\alpha,\quad |\alpha|\le N_*>n+1,
\]
uniformly in $N,h,k$. Hence the inverse Fourier transform of $a_{N,h,k}$ is uniformly bounded in $L^1$, implying
\[
\|S_{N,h,k}u\|_{L^p}\le C_\varphi\|u\|_{L^p},\qquad
S_{N,h,k}:=\mathcal F^{-1}[\widetilde\psi(\cdot/h-k)m_N],
\]
for $1\le p\le\infty$. Since $\Box_{h,k}\widetilde T_N g = S_{N,h,k}\Box_{h,k}g$, we get
\[
\|\Box_{h,k}\widetilde T_N g\|_{L^p}\le C_\varphi\|\Box_{h,k}g\|_{L^p}.
\]
Multiplying by $e^{\sigma h|k|_\infty}$ and taking supremum over $k$ gives $\|\widetilde T_N g\|_{\mathcal U^\sigma_{p,h}}\le C_\varphi\|g\|_{\mathcal U^\sigma_{p,h}}$. Taking $g=\varphi(D)f$ completes the proof.
\end{proof}

\begin{proposition}[A priori estimate for admissible multipliers]\label{prop-general-phi-apriori}
Let \(1\le p\le\infty\), and let \(u\in L^p(\R^n)\) solve \(\varphi(D)u=Vu\) with \(\|u\|_p=1\).  Fix the small parameter $\varepsilon\in(0,1)$ in Lemma \ref{lem-V-mult-weighted}, and let \(h_M=1+\delta M\) be the corresponding scale.  Then, for every \(\sigma>0\),
\begin{align}\label{eq-phi-apriori}
        \|u\|_{\mathcal U^\sigma_{p,h_M}}\le C e^{C\sigma N_\sigma},
\end{align}
where
\begin{align}\label{eq-phi-Nsigma}
        N_\sigma:=h_M+  \Phi_\varphi\big(2C(1+\mathcal V_\sigma)\big),\quad
               \mathcal V_\sigma:=C_\varepsilon A_0\exp\left((1+\varepsilon)(1-\kappa)(\sigma M)^{1/(1-\kappa)}+C_\varepsilon\sigma\right).
\end{align}
\end{proposition}

\begin{proof}
We first justify the a priori use of the analytic frequency-uniform norm.  As in the
proof of Lemma~\ref{lem-Up-apriori}, we work first with the truncated norm
\[
    \|f\|_{\mathcal U^{\sigma,L}_{p,h}}
    :=
    \sup_{k\in\mathbb Z^n}
    e^{\sigma h(|k|_\infty\wedge L)}
    \|\Box_{h,k}f\|_{L^p},
    \qquad L\ge1.
\]
The multiplication estimate for \(V\), the low-frequency estimate, and the
high-frequency estimate for \(\varphi(D)\) remain valid in this truncated norm, with
constants independent of \(L\).  Thus the absorption argument below is first performed
in \(\mathcal U^{\sigma,L}_{p,h_M}\), and the limit \(L\to\infty\) is taken at the end,
exactly as in the proof of Lemma~\ref{lem-Up-apriori}.

By Lemma~\ref{lem-V-mult-weighted},
\[
    \|Vu\|_{\mathcal U^{\sigma,L}_{p,h_M}}
    \le
    \mathcal V_\sigma
    \|u\|_{\mathcal U^{\sigma,L}_{p,h_M}},
\]
where
\begin{align}\label{eq-g-Nsigma-1} 
\mathcal V_\sigma
    :=
    C_\varepsilon A_0
    \exp\left((1+\varepsilon)(1-\kappa)(\sigma M)^{1/(1-\kappa)}
    +C_\varepsilon\sigma\right).
\end{align} 
Set
\[
    U_{\sigma,L}:=\|u\|_{\mathcal U^{\sigma,L}_{p,h_M}},
    \qquad
    Y_{\sigma,L}:=\|\varphi(D)u\|_{\mathcal U^{\sigma,L}_{p,h_M}}.
\]
Since \(\varphi(D)u=Vu\), we have
\[
    Y_{\sigma,L}\le \mathcal V_\sigma U_{\sigma,L}.
\]
The low-frequency estimate and Lemma~\ref{lem-phi-high} give, for \(N\ge h_M\),
\[
    U_{\sigma,L}
    \le
    C e^{C\sigma N}+C\varphi(N)^{-1}Y_{\sigma,L}
    \le
    C e^{C\sigma N}
    +
    C\varphi(N)^{-1}\mathcal V_\sigma U_{\sigma,L}.
\]
Choose
\[
    N_\sigma
    :=   h_M+\Phi_\varphi\big(2C(1+\mathcal V_\sigma)\big)
\]
such that
\(C\varphi(N_\sigma)^{-1}\mathcal V_\sigma\le \frac12\), and hence
\[
    U_{\sigma,L}
    \le
    C e^{C\sigma N_\sigma},
\]
uniformly in \(L\).  Passing to the limit \(L\to\infty\), as in
Lemma~\ref{lem-Up-apriori}, gives \eqref{eq-phi-apriori}.
\end{proof}

\begin{proof}[Proof of Theorem \ref{thm-general-phi}]
Let \(m=|\alpha|\) and
\[
        L_\varphi(m):=\log(e+\varphi(m)).
\]
Fix \(\theta>1\).  We shall prove the estimate with
\[
        B_{\theta,\varphi}:=\theta(1-\kappa)^{1-\kappa} M+K_{p,\theta,\varphi,A_0},
\]
where \(K_{p,\theta,\varphi,A_0}\) is chosen sufficiently large.  Choose small \(\varepsilon>0\) and \(\eta>0\) so that
\begin{align}\label{eq-g-Nsigma-2} 
(1+\varepsilon)(1+\eta)^{1/(1-\kappa)}\theta^{-1/(1-\kappa)}<1.    
\end{align}

For \(m\ge1\), set
\begin{align}\label{eq-sigma-phi-choice}
        \sigma_m:=\frac{1+\eta}{B_{\theta,\varphi}}L_\varphi(m)^{1-\kappa} .
\end{align}
Since \(B_{\theta,\varphi}\ge \theta(1-\kappa)^{1-\kappa} M\), we have
\begin{align}\label{eq-g-Nsigma-3}         (1+\varepsilon)(1-\kappa)(\sigma_m M)^{1/(1-\kappa)}
        \le (1+\varepsilon)(1+\eta)^{1/(1-\kappa)}\theta^{-1/(1-\kappa)}L_\varphi(m).
\end{align}
After increasing \(K_{p,\theta,\varphi,A_0}\), the additional term \(C_\varepsilon\sigma_m\) in \eqref{eq-g-Nsigma-1}  is absorbed into a small multiple of \(L_\varphi(m)\).  Hence by \eqref{eq-g-Nsigma-1}, \eqref{eq-g-Nsigma-2} and \eqref{eq-g-Nsigma-3}, there exists \(a_\theta\in(0,1)\) such that
\begin{align}\label{eq-Vsigma-phi-subpower}
        \mathcal V_{\sigma_m}\le C_{p,\theta,\varphi,A_0}(e+\varphi(m))^{a_\theta}.
\end{align}
By admissibility \eqref{assump-admissible-varphi-1} and \eqref{eq-phi-Nsigma}, there exist \(q=q_{\theta,\varphi}\in(0,1)\) and \(C_{p,\theta,\varphi,A_0}>0\) such that
\begin{align}\label{eq-Nsigma-phi-sublinear}
        N_{\sigma_m}\le h_M+C_{p,\theta,\varphi,A_0}(m+e)^q .
\end{align}
Consequently,
\begin{align}\label{eq-sigma-N-phi-subexp}
        \sigma_mN_{\sigma_m}=o(m).
\end{align}
Indeed, the second term in \eqref{eq-Nsigma-phi-sublinear} is subexponential by the choice of
\(q<1\), while the additional term \(h_M\) is harmless because
\(B_{\theta,\varphi}\ge \theta(1-\kappa)^{1-\kappa} M\), and hence
\(\sigma_mh_M\lesssim L_\varphi(m)^{1-\kappa}=o(m)\). Here we used Remark \ref{rem:admissible-symbol-condition}.

It remains to extract derivatives from the frequency-uniform analytic norm.   As in the proof of Theorem \ref{thm-fractional-Lp}, using \eqref{eq-phi-apriori}, we obtain
\begin{align}\label{eq-phi-derivative-before-choice}
        \|D^\alpha u\|_p
        \le C_{\eta,n}e^{C_n\sigma h_M} \frac 1 {(\sigma h_M)^n} (m+1)^{n-1}
        \frac{m!}{\sigma^m}e^{C\sigma N_\sigma}.
\end{align}
We now put $\sigma=\sigma_m$.  Since $h_M=1+\delta M$ with $\delta$ fixed, 
we have $\sigma_m h_M \lesssim \sigma_m M$ (up to a constant). 
Because $B_{\theta,\varphi}\ge \theta(1-\kappa)^{1-\kappa} M$, it follows that
\[
\sigma_m h_M \lesssim \frac{L_\varphi(m)^{1-\kappa}}{B_{\theta,\varphi}}M
\le C_\theta L_\varphi(m)^{1-\kappa}=o(m),
\]
where the last equality uses the admissibility condition which ensures 
$L_\varphi(m)=o(m^\rho)$ for every $\rho>0$ (see Remark \ref{rem:admissible-symbol-condition}).
Hence the factors $e^{C_n\sigma_m h_M}$, $(m+1)^{n-1}$, $(\sigma_m h_M)^{-n}$, and 
$e^{C\sigma_m N_{\sigma_m}}$ are all subexponential in $m$; they can be absorbed into 
a constant $C_{p,\theta,\varphi,A_0}(1+\eta)^m$ after increasing the constant 
to handle finitely many small $m$.  
Substituting \eqref{eq-sigma-phi-choice} into \eqref{eq-phi-derivative-before-choice} gives
\[
        \|D^\alpha u\|_p
        \le C_{p,\theta,\varphi,A_0}(1+\eta)^{m}\frac{m!}{\sigma_m^m}  = C_{p,\theta,\varphi,A_0}B_{\theta,\varphi}^m
        \frac{m!}{L_\varphi(m)^{(1-\kappa)m}} .
\]
The case $m=0$ follows from $\|u\|_p=1$, after enlarging the constant.  
This proves \eqref{eq-main-general-phi}.
\end{proof}

\begin{remark}[Stable-like nonlocal symbols] 
The admissible multiplier theorem is formulated for radial symbols \(\varphi(|\xi|)\). Nevertheless, the fractional estimate in Theorem~\ref{thm-fractional-Lp} for \(\Lambda^s\) extends, in the range \(0<s<2\) and \(1<p<\infty\), to non-radial translation-invariant operators that are \(L^p\)-elliptically comparable to \(\Lambda^s\). More precisely, it suffices to assume that \(L(D)\) satisfies \[ \|\Lambda^s w\|_{L^p(\mathbb R^n)} \le C_L \|L(D)w\|_{L^p(\mathbb R^n)} \] for all Schwartz functions \(w\). Then the proof of Theorem~\ref{thm-fractional-Lp} applies to solutions of \(L(D)u=Vu\), with constants depending additionally on \(C_L\) and on the structural constants of \(L\). Typical examples are the stable-like operators \[ L(D)=\psi_a(D),\qquad \psi_a(\xi) = \int_{\mathbb R^n} \bigl(1+iy\cdot \xi(1_{s\in (1,2)}+1_{s=1}1_{y\in B_1})-e^{iy\cdot\xi}\bigr) \frac{a(y)}{|y|^{n+s}}\, \d y,\]
\[
0<c_1\le a(y)\le c_2,\quad\text{and when $s=1$,}\,\int_{\partial B_r} a(y)\, \d \sigma_y=0\quad\forall r>0,\] 
for which the required comparison follows from the nonlocal elliptic \(L^p\)-theory of Dong--Kim~\cite{DongKim12}. The anisotropic stable operator \[ L(D)=\sum_{j=1}^n |D_j|^s \] is another example; its comparison follows from the \(L^p\)-theory for singular L\'evy measures, for instance by the method of Dong--Ryu~\cite{DongRyu26}. 
\end{remark}

\begin{proof}[Proof of the logarithmic-multiplier assertion in Theorem \ref{thm-general-phi}]
We finally verify explicitly that
\[
        \varphi(r)=\log(e+r)
\]
is admissible and yields the double logarithm in \eqref{eq-main-log-multiplier}.  In this case
\begin{align}\label{eq-logmult-Lphi-asymptotic}
        L_\varphi(m)=\log(e+\log(e+m)) .
\end{align}
Moreover,
\[
        \Phi_\varphi(t)=\exp(t)-e .
\]
Let \(a\in(0,1)\) and \(A\ge1\).  Then
\begin{align}\label{eq-logmult-inverse-growth}
        \Phi_\varphi\big(A(e+\varphi(r))^a\big)
        \le \exp\big(C_A(\log(e+r))^a\big).
\end{align}
Since \(a<1\), for every \(q\in(0,1)\) there exists \(C_{A,a,q}>0\) such that
\begin{align}\label{eq-logmult-sublinear-growth}
        \exp\big(C_A(\log(e+r))^a\big)\le C_{A,a,q}(r+e)^q .
\end{align}
Thus \eqref{assump-admissible-varphi-1} holds by \eqref{eq-logmult-inverse-growth}--\eqref{eq-logmult-sublinear-growth}.  Moreover, \eqref{assump-admissible-varphi-2} follows from the elementary bounds
\[
        |(r\partial_r)^j\log(e+r)|\le C_j\log(e+r),
        \qquad r>0,
\]
for every fixed \(j\).  Hence \(\varphi(r)=\log(e+r)\) is admissible.  Applying Theorem \ref{thm-general-phi} and \eqref{eq-logmult-Lphi-asymptotic} gives
\[
        \|D^\alpha u\|_p
        \le C_{p,\theta,A_0}
        \left[\theta(1-\kappa)^{1-\kappa} M+K_{p,\theta,A_0}\right]^{|\alpha|}
        \frac{|\alpha|!}{\big(\log(e+\log(e+|\alpha|))\big)^{(1-\kappa)|\alpha|}}.
\]
This is the desired double-logarithmic analogue of the fractional estimate.
\end{proof}

\begin{remark}
For \(\varphi(r)=r^s\), the admissibility conditions are immediate, and one has \(L_\varphi(m)=\log(e+m^s)\sim s\log(m+e)\).  Therefore, Theorem \ref{thm-general-phi} gives
\[
        \frac{\bigl((1-\kappa)^{1-\kappa} M\bigr)^m}{L_\varphi(m)^{(1-\kappa)m}}
        \sim
        \left(M\left(\frac{1-\kappa}{s}\right)^{1-\kappa}\right)^m
        \frac1{\big(\log(m+e)\big)^{(1-\kappa)m}},
\]
which recovers Theorem \ref{thm-fractional-Lp}.  The logarithmic multiplier is slower than any power, and this is exactly why a double logarithm replaces the single logarithm.
\end{remark}

\section{Sharpness and optimality discussion}\label{sec-sharpness-discussion}

This section discusses the sharpness of the logarithmic factors and of the leading scale in
the preceding estimates. The sharpness results are based on periodic Fourier constructions supported on
one side of the frequency spectrum, namely on nonnegative Fourier modes. We give the fractional construction in detail, since it detects both the
optimal logarithmic power and the optimal leading coefficient in front of the type parameter
\(M\). We then record the logarithmic-multiplier analogue, where the same frequency-recursion
mechanism shows the optimality of the double-logarithmic scale for the model multiplier
\(\varphi(r)=\log(e+r)\).

\subsection{Sharpness of the fractional estimate}

We now show that, for every fractional order $s>0$, the logarithmic power and the
leading coefficient in Theorem \ref{thm-fractional-Lp} are forced by the equation.
The sharpness results are based on periodic positive-frequency Fourier
constructions.  This is enough to rule
out any uniform improvement of the fractional estimate.

\begin{proposition}[Optimality of the leading coefficient for fractional powers]
\label{prop-periodic-sharp-leading-coefficient-fractional}
Let $s>0$ and $0\le\kappa<1$.  In the class of
$2\pi$-periodic equations $\Lambda^s u=Vu$ with
$\|V^{(m)}\|_{L^\infty(\mathbb T)}\le C_0M^m(m!)^\kappa$, the following two
improvements of Theorem \ref{thm-fractional-Lp} are impossible.
\begin{enumerate}
\item If \(b<\left(\frac{1-\kappa}{s}\right)^{1-\kappa}\), then there are no constants \(C,K>0\), independent of \(M\), such that every normalized solution in this class, say with \(\|u\|_{L^\infty(\mathbb T)}=1\), satisfies
\begin{align}\label{eq-false-leading-coeff-fractional}
  \|u^{(m)}\|_{L^\infty(\mathbb T)}
  \le C (bM+K)^m
  \frac{m!}{\bigl(\log(m+e)\bigr)^{(1-\kappa)m}},
  \qquad m=0,1,2,\ldots .
\end{align}
\item For every $\varepsilon>0$, the denominator
$\bigl(\log(m+e)\bigr)^{(1-\kappa)m}$ cannot in general be replaced with
$\bigl(\log(m+e)\bigr)^{(1-\kappa+\varepsilon)m}$.
\end{enumerate}
\end{proposition}

\begin{proof}
We work on \(\mathbb T=\mathbb R/2\pi\mathbb Z\), where
\[
  \Lambda^s e^{inx}=|n|^s e^{inx}.
\]

\smallskip\noindent{\it Step 1: a one-sided periodic model.}
First assume \(0<\kappa<1\).  Let \(M\ge2\) be a large parameter.  Define a
one-sided periodic potential
\begin{align*}
  V_M(x):=\sum_{k=1}^\infty v_k e^{ikx},
  \qquad
  v_k:=\frac{\eta}{M}
  \exp\left[-\kappa\left(\frac{k}{M}\right)^{1/\kappa}\right],
\end{align*}
where \(0<\eta\ll1\) is fixed.  A standard Laplace estimate gives
\begin{align*}
  \|V_M^{(m)}\|_{L^\infty(\mathbb T)}
  \le
  \sum_{k=1}^\infty k^m v_k
  \le
  C_\eta M^m(m!)^\kappa,
  \qquad m=0,1,2,\ldots .
\end{align*}
Indeed, after the change of variables \(k=Mr\), the sum is controlled by
\[
  M^m\int_0^\infty r^m
  \exp(-\kappa r^{1/\kappa})\,dr
  \lesssim
  M^m(m!)^\kappa .
\]

We construct \(u_M\) by a positive Fourier recursion.  Let \(a_0=1\), and for
\(n\ge1\) set
\begin{align}\label{eq-leading-coeff-recursion}
  n^s a_n
  =
  \sum_{k=1}^n v_k a_{n-k}.
\end{align}
Define
\[
  u_M(x):=\sum_{n=0}^\infty a_n e^{inx}.
\]
Then \eqref{eq-leading-coeff-recursion} is exactly the Fourier coefficient identity
for
\[
  \Lambda^s u_M=V_Mu_M.
\]

\smallskip\noindent{\it Step 2: smoothness of the constructed solution.}
We show that \(u_M\) is smooth.  Let
\begin{align*}
  A_n:=\exp\left[
    -c\frac{n}{M}\bigl(\log(n+e)\bigr)^{1-\kappa}
  \right] 
\end{align*}
with some small  \(c>0\) determined later. It is easy to see that for some $C=C(\kappa)>0$
\begin{align}\label{eq-sharp-frac-ratio-An}
  \frac{A_{n-k}}{A_n}
  \le
  \exp\left[
    Cc\frac{k}{M}\bigl(\log(n+e)\bigr)^{1-\kappa}
  \right],
  \qquad 1\le k\le n.
\end{align}
Indeed, this follows by applying the mean value theorem to $F(x)=x(\log(x+e))^{1-\kappa}$. Since $A_n=\exp[-(c/M)F(n)]$, the ratio $A_{n-k}/A_n$ is controlled by the increment $F(n)-F(n-k)$. The derivative of $F$ is bounded by $C_\kappa(\log(n+e))^{1-\kappa}$ on $[0,n]$, and hence $F(n)-F(n-k)\le C_\kappa k(\log(n+e))^{1-\kappa}$. This proves \eqref{eq-sharp-frac-ratio-An}. 

Hence, using \eqref{eq-sharp-frac-ratio-An},
\begin{align*}
  n^{-s}\sum_{k=1}^n v_k A_{n-k}
  \le
  A_n n^{-s}\frac{\eta}{M}
  \sum_{k=1}^n
  \exp\left[
    -\kappa\left(\frac{k}{M}\right)^{1/\kappa}
    +
    Cc\frac{k}{M}\bigl(\log(n+e)\bigr)^{1-\kappa}
  \right].
\end{align*}
Young's inequality gives
\[
Cc\,x \bigl(\log(n+e)\bigr)^{1-\kappa}
\le
\frac{\kappa}{2}x^{1/\kappa}
+C c^{1/(1-\kappa)}\log(n+e),
\qquad x\ge0 .
\]
Applying this with $x=k/M$, we obtain
\begin{multline*}
\frac1M\sum_{k=1}^n
\exp\left[
-\kappa\left(\frac{k}{M}\right)^{1/\kappa}
+
Cc\frac{k}{M}\bigl(\log(n+e)\bigr)^{1-\kappa}
\right]\\
\le
e^{C c^{1/(1-\kappa)}\log(n+e)}
\frac1M\sum_{k=1}^\infty
\exp\left[
-\frac{\kappa}{2}\left(\frac{k}{M}\right)^{1/\kappa}
\right]
\le C n^{C c^{1/(1-\kappa)}}.    
\end{multline*} 
Therefore,
\begin{align}\label{eq-sharp-frac-induction-majorant}
n^{-s}\sum_{k=1}^n v_k A_{n-k}
\le
C\eta n^{-s+C c^{1/(1-\kappa)}} A_n
\le A_n
\end{align}
after choosing $c>0$ so small that $C c^{1/(1-\kappa)}<s$, and then choosing
$\eta>0$ sufficiently small. By induction from
\eqref{eq-leading-coeff-recursion} and \eqref{eq-sharp-frac-induction-majorant},
\begin{align*}
a_n\le A_n,
\qquad n=0,1,2,\ldots .
\end{align*}
Thus $u_M\in C^\infty(\mathbb T)\subset L^\infty(\mathbb T)$.

We shall use this solution only up to multiplication by a constant.  Indeed, the
equation \(\Lambda^s u=Vu\) is homogeneous in \(u\).  Hence, if the statement to
be contradicted is formulated for normalized solutions, we may replace \(u_M\) with
\[
        \widetilde u_M:=\frac{u_M}{\|u_M\|_{L^\infty(\mathbb T)}} .
\]
This only multiplies all Fourier coefficients by the fixed factor
\(\|u_M\|_{L^\infty(\mathbb T)}^{-1}\).  Equivalently, any upper bound obtained
for the normalized solution gives the same upper bound for the coefficients
\(a_n=\widehat u_M(n)\), up to a harmless multiplicative constant depending on
\(M\).  Since \(M\) is fixed before the limiting sequence \(n=n_L\to\infty\) is
taken, this constant has no effect on the exponential comparison below.

\smallskip\noindent{\it Step 3: a lower bound along a selected subsequence.}
We now prove the lower bound on selected Fourier coefficients.  Let \(L\to\infty\),
and choose
\begin{align*}
  q_L:=\left\lfloor
    M\left(\frac{sL}{1-\kappa}\right)^\kappa
  \right\rfloor,
  \qquad
  \ell_L:=\lfloor e^L\rfloor,
  \qquad
  n_L:=q_L\ell_L .
\end{align*}
Iterating the single jump \(k=q_L\) in the recursion gives
\begin{align}\label{eq-sharp-frac-single-jump}
  a_{j q_L}
  \ge
  \frac{v_{q_L}}{(j q_L)^s}
  a_{(j-1)q_L},
  \qquad j=1,2,\ldots,\ell_L.
\end{align}
Therefore, by iterating \eqref{eq-sharp-frac-single-jump},
\begin{align*}
  a_{n_L}
  \ge
  \frac{v_{q_L}^{\ell_L}}
       {q_L^{s\ell_L}(\ell_L!)^s}.
\end{align*}
Using Stirling's formula, we obtain
\[
  -\log a_{n_L}
  \le
  \ell_L
  \left[
    \kappa\left(\frac{q_L}{M}\right)^{1/\kappa}
    +s\log(q_L\ell_L)
    +O(\log M+\log L)
  \right].
\]
Since
\begin{align*}
  \log n_L=L+O(\log M+\log L)
\end{align*}
and
\begin{align*}
  q_L
  =
  M\left(\frac{sL}{1-\kappa}\right)^\kappa(1+o(1)),
\end{align*}
we get
\begin{align}\label{eq-leading-coeff-lower-bound}
  a_{n_L}
  \ge
  \exp\left[
    -\left(
      \left(\frac{s}{1-\kappa}\right)^{1-\kappa}+o(1)
    \right)
    \frac{n_L}{M}
    \bigl(\log n_L\bigr)^{1-\kappa}
  \right]
\end{align}
as \(L\to\infty\).  The constant in the exponent follows from
\[
  \frac{
    \kappa(q_L/M)^{1/\kappa}+s\log n_L
  }{q_L}
  =
  \left(\frac{s}{1-\kappa}\right)^{1-\kappa}
  \frac{(\log n_L)^{1-\kappa}}{M}
  (1+o(1)).
\]

\smallskip\noindent{\it Step 4: contradiction to improved estimates.}
Suppose now that \eqref{eq-false-leading-coeff-fractional} held with some
\begin{align}\label{eq-leading-smaller}
  b<\left(\frac{1-\kappa}{s}\right)^{1-\kappa} .    
\end{align} 
Since \(\widehat{u_M}(n)=a_n\), we would have, for every \(m\ge0\),
\begin{align*}
  a_n
  \le
  C\frac{\|u_M^{(m)}\|_{L^\infty}}{n^m}
  \le
  C
  \left(\frac{bM+K}{n}\right)^m
  \frac{m!}{\bigl(\log(m+e)\bigr)^{(1-\kappa)m}}.
\end{align*}
Choosing
\begin{align*}
  m\sim
  \frac{n}{bM+K}\bigl(\log n\bigr)^{1-\kappa},
\end{align*}
yields
\begin{align*}
  a_n
  \le
  C
  \exp\left[
    -\left(\frac{1}{bM+K}+o(1)\right)
    n(\log n)^{1-\kappa}
  \right].
\end{align*}
Since \eqref{eq-leading-smaller},
we may choose \(M\) sufficiently large so that
\[
  \frac{1}{bM+K}
  >
  \left(\frac{s}{1-\kappa}\right)^{1-\kappa}\frac{1+\delta}{M}
\]
for some \(\delta>0\).  Applying the last upper bound to \(n=n_L\) contradicts
\eqref{eq-leading-coeff-lower-bound} for \(L\) sufficiently large.  Hence no
estimate with leading coefficient \(b\) can hold.

The same example also rules out any improvement of the logarithmic power.  Indeed,
if for some \(\varepsilon>0\) one had
\[
  \|u_M^{(m)}\|_{L^\infty(\mathbb T)}
  \le
  C B^m
  \frac{m!}{\bigl(\log(m+e)\bigr)^{(1-\kappa+\varepsilon)m}},
  \qquad m=0,1,2,\ldots .
\]
Using the Fourier coefficient estimate \[|a_n|\le \|u_M^{(m)}\|_{L^\infty}/|n|^m\] and Stirling's formula \(m!\sim m^{1/2}(m/e)^m\), we obtain 
\[|a_n|\le C\left(\frac{Bm}{e|n|}\right)^m (\log(m+e))^{-(1-\kappa+\varepsilon)m}.\] 
Choosing \(m=\lfloor c\,n(\log n)^{1-\kappa+\varepsilon}\rfloor\), with \(c>0\) sufficiently small, we have
\[
  a_n
  \le
  C\exp\left[
    -c' n\bigl(\log n\bigr)^{1-\kappa+\varepsilon}
  \right]
\]
for all large \(n\).  This contradicts
\eqref{eq-leading-coeff-lower-bound} along \(n=n_L\).

\smallskip\noindent{\it Step 5: the endpoint \(\kappa=0\).}
It remains to treat the endpoint \(\kappa=0\).  Let
\(N=\lfloor M\rfloor\) and take
\[
  V_M(x):=  e^{iNx}.
\]
Then
\[
  \|V_M^{(m)}\|_{L^\infty(\mathbb T)}
  \le
    M^m,
  \qquad m=0,1,2,\ldots .
\]
Define
\[
  a_0=1,
  \qquad
  a_{jN}:=\frac{1}{N^{sj}(j!)^s},
  \qquad
  a_n:=0
  \quad\text{if }N\nmid n,
\]
and set
\[
  u_M(x):=\sum_{j=0}^\infty a_{jN}e^{ijNx}.
\]
Then
\[
  (jN)^s a_{jN}=  a_{(j-1)N},
\]
so again
\[
  \Lambda^s u_M=V_Mu_M.
\]
For \(n=jN\), Stirling's formula gives
\[
  a_n
  =
  \frac{1}{N^{sj}(j!)^s}
  \ge
  \exp\left[
    -\left(s+o(1)\right)
    \frac{n}{M}\log n
  \right].
\]
If an estimate of the form
\[
  \|u^{(m)}\|_{L^\infty}
  \le
  C(bM+K)^m
  \frac{m!}{\log^m(m+e)}
\]
held with \(b<1/s\), then the same Fourier coefficient optimization would imply
\[
  a_n
  \le
  C
  \exp\left[
    -\left(\frac{1}{bM+K}+o(1)\right)n\log n
  \right].
\]
For \(M\) large, this contradicts the lower bound above, since \(b<1/s\).

Likewise, if the denominator \(\log^m(m+e)\) were replaced with
\(\log^{(1+\varepsilon)m}(m+e)\), then one would obtain
\[
  a_n
  \le
  C\exp\left[-c n(\log n)^{1+\varepsilon}\right],
\]
again contradicting
\[
  a_n\ge \exp[-C n\log n]
\]
along the subsequence \(n=jN\).  This completes the proof.
\end{proof}

\subsection{Sharpness of the logarithmic multiplier}

We finally record that the double-logarithmic scale in Theorem \ref{thm-general-phi}
is optimal for the model multiplier \(\varphi(r)=\log(e+r)\).  The construction is
parallel to the fractional sharpness example, so we only indicate the necessary changes.
\begin{proposition}[A periodic obstruction for the double logarithm]
\label{prop-periodic-sharp-loglog-power}
Let $0\le\kappa<1$.  For every $\varepsilon>0$, there are
a $2\pi$-periodic potential $V$ and a nonzero bounded $2\pi$-periodic solution of
$\log(e+|D|)u=Vu$ such that
\[
  \|V^{(m)}\|_{L^\infty(\mathbb T)}\le C A^m(m!)^\kappa,
  \qquad m=0,1,2,\ldots,
\]
but no constants $B,C_1<\infty$ can make
\begin{align}\label{eq-false-loglog-kappa-improved}
  \|u^{(m)}\|_{L^\infty(\mathbb T)}
  \le C_1B^m
  \frac{m!}{\bigl(\log(e+\log(e+m))\bigr)^{(1-\kappa+\varepsilon)m}},
  \qquad m=0,1,2,\ldots .
\end{align}
Thus the exponent $1-\kappa$ in the double-logarithmic denominator of
\eqref{eq-main-log-multiplier} cannot be increased.
\end{proposition}

\begin{proof}
We only indicate the changes from Proposition
\ref{prop-periodic-sharp-leading-coefficient-fractional}.  On $\mathbb T$,
\[
  \log(e+|D|)e^{inx}=\log(e+|n|)e^{inx}.
\]
For $0<\kappa<1$, take
\[
  V(x)=1+\sum_{k=1}^\infty v_ke^{ikx},
  \qquad v_k=\eta\exp(-\kappa k^{1/\kappa})
\]
with $0<\eta\ll1$.  Then $V$ satisfies the required derivative bounds.  Define
$a_0=1$ and, for $n\ge1$,
\begin{align}\label{eq-logD-recursion}
  \bigl(\log(e+n)-1\bigr)a_n=
  \sum_{k=1}^n v_ka_{n-k}.
\end{align}
The resulting series $u(x)=\sum_{n\ge0}a_ne^{inx}$ solves
$\log(e+|D|)u=Vu$.  As in the fractional proof, the comparison sequence
$A_n=\exp[-c n(\log(e+\log(e+n)))^{1-\kappa}]$ gives $a_n\le A_n$ after choosing
$c$ and $\eta$ small; hence $u$ is smooth.

The obstruction comes from the same single-jump iteration.  Let $L\to\infty$,
choose
\[
  q_L=\left\lfloor (L/(1-\kappa))^\kappa\right\rfloor,
  \qquad n_L=q_L\ell_L,
  \qquad \ell_L\sim e^{e^L}/q_L .
\]
Iterating the jump $k=q_L$ in \eqref{eq-logD-recursion} gives
\begin{align}\label{eq-logD-lower-subsequence}
  a_{n_L}
  \ge
  \exp\left[-\left((1-\kappa)^{-(1-\kappa)}+o(1)\right)
  n_L(\log\log n_L)^{1-\kappa}\right].
\end{align}
On the other hand, \eqref{eq-false-loglog-kappa-improved} would imply, by the
standard Fourier-coefficient optimization,
\[
  a_n\le C\exp[-c n(\log\log n)^{1-\kappa+\varepsilon}]
\]
for all large $n$, contradicting \eqref{eq-logD-lower-subsequence} along
$n=n_L$.

For $\kappa=0$, take $V(x)=1+\eta e^{ix}$ and
\[
  a_n=\frac{\eta^n}{\prod_{j=1}^n(\log(e+j)-1)}.
\]
Then $\log(e+|D|)u=Vu$ and $a_n\ge\exp[-C n\log\log(e+n)]$, while an estimate
with exponent $1+\varepsilon$ in the double logarithm would force
$a_n\le C\exp[-c n(\log\log n)^{1+\varepsilon}]$.  This contradiction completes
the proof.
\end{proof}

\begin{remark}
The logarithmic-multiplier example is the same one-sided Fourier mechanism as in
the fractional sharpness proof.  The only structural change is the denominator in
the recursion: $n^s$ is replaced with
  $\log(e+n)-1\sim\log n$.
This replacement changes the critical coefficient decay from
\(\exp[-c n(\log n)^{1-\kappa}]\) to
\(\exp[-c n(\log\log n)^{1-\kappa}]\), and hence changes the sharp denominator from a
single logarithm to the double logarithm
\(\log(e+\log(e+m))\).
\end{remark}

\section{Invariant ultra-analytic classes}\label{sec-invariant-classes}

In this section we make precise the invariant-class question introduced in the
Introduction.  The preceding results show that the fractional equation usually
weakens the ultra-analytic scale of the coefficient. We
formulate a transfer criterion with explicit Fourier-side hypotheses, and then
apply it to the concrete weights used in this paper.

\subsection{Weighted ultra-analytic classes}

We first fix the notation.  All classes in this section are {\it total-order}
classes: the derivative of order \(|\alpha|\) is measured with \(|\alpha|!\), not
with the coordinate factorial \(\alpha!\).  This is the same total-order convention
used in Theorems \ref{thm-fractional-Lp} and \ref{thm-general-phi}.

\begin{definition}[Admissible weights]
A sequence \(L=\{L_m\}_{m\ge0}\) is called an {\it admissible ultra-analytic weight}
if
\[
 L_0=L_1=1,
 \qquad
 1\le L_2\le L_3\le\cdots,
 \qquad
 L_m\to\infty\quad\text{as }m\to\infty.
\]
The monotonicity assumption is only a normalization.  The condition
\(L_m\to\infty\) is what makes the corresponding class smaller than the usual
analytic class.
\end{definition}

\begin{definition}[The spaces \(\mathcal U_L^p(A)\) and \(\mathcal U_L^p\)]
Let \(1\le p\le\infty\), let \(A>0\), and let \(L\) be an admissible weight.  We define
\begin{align*}
 \|f\|_{\mathcal U_L^p(A)}
 :=
 \sup_{\alpha\in\mathbb N^n}
 \frac{L_{|\alpha|}^{|\alpha|}}{A^{|\alpha|}|\alpha|!}
 \|D^\alpha f\|_{L^p(\mathbb R^n)}.
\end{align*}
Here we use the convention \(L_0^0=A^0=1\).  Then
\[
 \mathcal U_L^p(A)
 :=
 \{f\in C^\infty(\mathbb R^n):\|f\|_{\mathcal U_L^p(A)}<\infty\},
\]
and the qualitative class is
\begin{align*}
 \mathcal U_L^p
 :=
 \bigcup_{A>0}\mathcal U_L^p(A).
\end{align*}
Thus \(f\in\mathcal U_L^p\) means that there exist constants \(A,C>0\) such that
\begin{align}\label{eq-U-L-equivalent-bound-new}
 \|D^\alpha f\|_{L^p(\mathbb R^n)}
 \le
 C A^{|\alpha|}
 \frac{|\alpha|!}{L_{|\alpha|}^{|\alpha|}},
 \qquad \alpha\in\mathbb N^n.
\end{align}
\end{definition}

\begin{example}[Standard weights]
If \(L_m=m^\nu\), \(0<\nu\le1\), then \eqref{eq-U-L-equivalent-bound-new} gives
\[
 \|D^\alpha f\|_{L^p}
 \le
 C A^{|\alpha|}\frac{|\alpha|!}{|\alpha|^{\nu|\alpha|}}.
\]
By Stirling's formula this is equivalent, up to changing \(A\), to the Gevrey-type
ultra-analytic estimate
\[
 \|D^\alpha f\|_{L^p}
 \lesssim
 A^{|\alpha|}(|\alpha|!)^{1-\nu}.
\]
If \(L_m=(\log(m+e))^\gamma\), then \(\mathcal U_L^p\) is the logarithmic
ultra-analytic class
\[
 \|D^\alpha f\|_{L^p}
 \le
 C A^{|\alpha|}
 \frac{|\alpha|!}{\log^{\gamma|\alpha|}(|\alpha|+e)}.
\]
\end{example}

\subsection{A transfer criterion for ultra-analytic scales}

For a weight \(L\) and a parameter \(A>0\), define the associated function
\[
 \Omega_{L,A}(R)
 :=
 \sup_{N\ge1}
 \left\{
 N\log\left(\frac{RL_N}{A}\right)-\log(N!)
 \right\},
 \qquad R\ge1,
\]
and its Legendre-type dual
\[
 H_{L,A}(\sigma)
 :=
 \sup_{R\ge1}\{\sigma R-\Omega_{L,A}(R)\},
 \qquad \sigma>0.
\]
These two functions encode, respectively, the annular frequency decay obtained from
derivative bounds and the cost of multiplying by \(V\) in analytic frequency-uniform
norms.

The following definition isolates the precise hypotheses needed for the transfer
argument.  It is deliberately stated as a criterion rather than as a characterization
of all possible weights.

\begin{definition}[Transfer-admissible weights]\label{def-transfer-admissible}
Let \(L=\{L_m\}_{m\ge0}\) be an admissible weight.  We say that \(L\) is
{\it transfer-admissible} if, for every \(A>0\), the following properties hold after
changing the constants by amounts depending only on \(L\), \(A\), and the fixed
parameters of the decomposition.

\smallskip
\noindent
\((T1)\) If \(V\in\mathcal U_L^\infty(A)\), then its smooth dyadic annular pieces
\(V_q\), localized to \(|\xi|_\infty\simeq 2^q\), satisfy
\begin{align}\label{eq-transfer-admissible-T1}
 \|V_q\|_{L^\infty}
 \le
 C_A\exp\bigl(-c_A\Omega_{L,A}(c_A2^q)\bigr),
 \qquad q\ge0.
\end{align}

\smallskip
\noindent
\((T2)\) The decay in \((T1)\) implies the summability estimate
\begin{align}\label{eq-transfer-admissible-T2}
 \sum_{q\ge0}
 \exp\bigl(C\sigma 2^q-c\Omega_{L,A}(c2^q)\bigr)
 \le
 C_A\exp\bigl(C_AH_{L,A}(C_A\sigma)+C_A\sigma\bigr),
 \qquad \sigma>0.
\end{align}
The harmless polynomial overlap factors coming from the frequency-uniform blocks are
included in the constants in \eqref{eq-transfer-admissible-T2}.

\smallskip
\noindent
\((T3)\) Define the transferred weight by
\begin{align}\label{def-transferred-weight-new}
 \widetilde L_m
 :=
 \sup\left\{
 \sigma>0:
 C_AH_{L,A}(C_A\sigma)+C_A\sigma\le c_s\log(m+e)
 \right\},
 \qquad m\ge2,
\end{align}
where \(c_s>0\) is chosen sufficiently small. We also assume that the sequence
\(\widetilde L=\{\widetilde L_m\}_{m\ge2}\) is admissible.  If \(\sigma_m\le c\widetilde L_m\), then the
absorption threshold
\begin{align}\label{eq-transfer-N-sigma-general}
 N_\sigma
 :=
 C_A\exp\left(\frac{C_AH_{L,A}(C_A\sigma)+C_A\sigma}{s}\right)
\end{align}
satisfies
\begin{align}\label{eq-transfer-sublinear-condition}
 \sigma_m N_{\sigma_m}=o(m),
 \qquad m\to\infty.
\end{align}
\end{definition}

\begin{remark}
For the standard weights used below, conditions \((T1)\)--\((T3)\) are verified by
explicit one-dimensional optimizations.  For example, when $\nu\in(0,1)$, \(L_m=m^\nu\) gives \(\Omega_{L,A}(R)\simeq (R/A)^{1/(1-\nu)}\),
\(H_{L,A}(\sigma)\simeq (A\sigma)^{1/\nu}\) and hence
\(\widetilde L_m\simeq (\log(m+e))^\nu\).  Similarly,
\(L_m=(\log(m+e))^\gamma\) transfers to a double-logarithmic weight, while the
iterated-logarithm weight in Example \ref{ex-log-star-weight} is stable under this
transfer.  
\end{remark}

\begin{theorem}[Transfer criterion for ultra-analytic scales]\label{thm-transfer-scale}
Let \(s>0\), \(1\le p\le\infty\), and let \(u\in L^p(\mathbb R^n)\) solve
\[
 \Lambda^s u=Vu,
 \qquad
 \|u\|_{L^p}=1.
\]
Assume that \(L\) is transfer-admissible and that
\(V\in\mathcal U_L^\infty(A)\) for some \(A>0\).  Let \(\widetilde L\) be the
transferred weight defined by \eqref{def-transferred-weight-new}.  Then
\(u\in\mathcal U_{\widetilde L}^p\).  Equivalently, there exist constants \(B,C>0\)
such that
\begin{align}\label{eq-transfer-derivative-new}
 \|D^\alpha u\|_{L^p}
 \le
 C B^{|\alpha|}
 \frac{|\alpha|!}{\widetilde L_{|\alpha|}^{|\alpha|}},
 \qquad \alpha\in\mathbb N^n.
\end{align}
\end{theorem}

\begin{proof}
The proof is a direct abstraction of the preceding low--high frequency
argument. By \((T1)\), the derivative bounds on \(V\) imply the
annular frequency decay \eqref{eq-transfer-admissible-T1}.  Combining this decay
with the localized multiplication estimate in the frequency-uniform analytic norm
and using \((T2)\), we obtain
\begin{align}\label{eq-transfer-multiplication-bound}
 \|Vf\|_{\mathcal U^\sigma_{p,h}}
 \le
 C_A\exp\bigl(C_AH_{L,A}(C_A\sigma)+C_A\sigma\bigr)
 \|f\|_{\mathcal U^\sigma_{p,h}}.
\end{align}
Here \(h\ge1\) is the fixed frequency-uniform scale used in the proof, and all
constants are allowed to depend on the quantitative \(\mathcal U_L^\infty(A)\)-norm
of \(V\).

Set
\[
 U_\sigma:=\|u\|_{\mathcal U^\sigma_{p,h}},
 \qquad
 Y_\sigma:=\|\Lambda^s u\|_{\mathcal U^\sigma_{p,h}}.
\]
From \(\Lambda^s u=Vu\) and \eqref{eq-transfer-multiplication-bound},
\begin{align}\label{eq-transfer-Y-bound}
 Y_\sigma
 \le
 C_A\exp\bigl(C_AH_{L,A}(C_A\sigma)+C_A\sigma\bigr)U_\sigma.
\end{align}
The usual low--high frequency decomposition gives
\begin{align}\label{eq-transfer-low-high-system}
 U_\sigma
 \le
 C e^{C\sigma N}
 + C N^{-s}Y_\sigma.
\end{align}
Choose \(N=h+N_\sigma\) with $N_\sigma$ as in \eqref{eq-transfer-N-sigma-general}.  Then the second term in \eqref{eq-transfer-low-high-system}
is absorbed by the left-hand side, and hence
\begin{align}\label{eq-transfer-apriori-U}
 U_\sigma\le C\exp(C\sigma N_\sigma).
\end{align}

Let \(m=|\alpha|\ge1\).  Choose \(\sigma_m\le c\widetilde L_m\) so that
\begin{align}\label{eq-transfer-sigma-choice}
 C_AH_{L,A}(C_A\sigma_m)+C_A\sigma_m\le c_s\log(m+e).
\end{align}
Then \eqref{eq-transfer-N-sigma-general} gives \(N_{\sigma_m}\le C(m+e)^q\) for
some \(q<1\), provided \(c_s\) is sufficiently small.  By \((T3)\),
\[ \sigma_mN_{\sigma_m}=o(m). \]
The derivative extraction from the frequency-uniform analytic norm is the same as
in the proof of Theorem \ref{thm-general-phi}: for every fixed \(\eta>0\),
\[
 \|D^\alpha u\|_p
 \le
 C_{\eta,n}e^{C\sigma h}(m+1)^{n-1}\frac{1}{(\sigma h)^n} 
 \frac{m!}{\sigma^m}
 \|u\|_{\mathcal U^\sigma_{p,h}}.
\]
Applying this with \(\sigma=\sigma_m\), and using
\eqref{eq-transfer-apriori-U}, we obtain
\[
 \|D^\alpha u\|_p
 \le
 C_{\eta,n}e^{C\sigma_mh}(m+1)^{n-1}\frac{1}{(\sigma_m h)^n}  
 \frac{m!}{\sigma_m^m}e^{C\sigma_mN_{\sigma_m}}.
\]
The factors \(e^{C\sigma_mh}\), \((m+1)^{n-1}\), \((\sigma_m h)^{-n}\), and
\(e^{C\sigma_mN_{\sigma_m}}\) are subexponential in \(m\).  They can therefore be
absorbed into a larger base constant \(B^m\).  Since
\(\sigma_m\simeq\widetilde L_m\), this gives \eqref{eq-transfer-derivative-new}.
The case \(m=0\) follows from \(\|u\|_p=1\).  The theorem is proved.
\end{proof}
 
As direct consequences of the transfer principle, we have the following results.

\begin{corollary}[Polynomial ultra-analytic scales]
Let \(L_m=m^\nu\), \(0<\nu\le1\).  This weight is transfer-admissible and
\(\widetilde L_m\simeq (\log(m+e))^\nu\).  Hence, if
\(V\in\mathcal U_L^\infty\), then every normalized solution of \(\Lambda^s u=Vu\)
satisfies
\begin{align}\label{eq-cor-polynomial-scale-bound}
 \|D^\alpha u\|_p
 \le
 C B^{|\alpha|}
 \frac{|\alpha|!}{(\log(|\alpha|+e))^{\nu|\alpha|}}.
\end{align}
This recovers the logarithmic ultra-analytic estimate of Theorem
\ref{thm-fractional-Lp}.
\end{corollary}

\begin{corollary}[Logarithmic-to-double-logarithmic transfer]
Let \(L_m=(\log(m+e))^\gamma\), \(\gamma>0\).  This weight is transfer-admissible and
\(\widetilde L_m\simeq(\log\log(m+e^e))^\gamma\).  Hence, if
\(V\in\mathcal U_L^\infty\), then the solution satisfies
\begin{align*}
 \|D^\alpha u\|_p
 \le
 C B^{|\alpha|}
 \frac{|\alpha|!}{(\log\log(|\alpha|+e^e))^{\gamma|\alpha|}}.
\end{align*}
Thus a logarithmic ultra-analytic coefficient generally gives a double-logarithmic
ultra-analytic solution.
\end{corollary}

\subsection{An invariant family generated by logarithmically stable weights}
Theorem~\ref{thm-transfer-scale} shows that, under the fractional equation, the coefficient 
scale \(L_m\) is essentially transferred to the solution scale \(L_{\lfloor c\log(m+e)\rfloor}\). 
Therefore, a scale \(L\) yields an invariant class when it is stable, up to harmless 
constants, under the logarithmic shift \(m\mapsto\log m\). The following definition makes 
this precise.

\begin{definition}[Logarithmic stability]
A transfer-admissible weight \(L\) is called {\it logarithmically stable} if its
transferred weight satisfies
\begin{align}\label{eq-log-stability-new}
 \widetilde L_m\ge c_1L_m,
 \qquad m\ge m_0,
\end{align}
for some constants \(c_1>0\) and \(m_0\ge2\).
\end{definition}

\begin{corollary}[Logarithmically stable classes are invariant]\label{cor-log-stable}
Assume that \(L\) is logarithmically stable.  If \(V\in\mathcal U_L^\infty\), then
every normalized solution of \(\Lambda^s u=Vu\) satisfies
\[
 u\in\mathcal U_L^p.
\]
\end{corollary}

\begin{example}[The iterated-logarithm weight]\label{ex-log-star-weight}
We now give a detailed example of a logarithmically stable weight.  Fix a threshold
\(C_*>1\).  For \(m>C_*\), define
\[
 \log^*m
 :=
 \min\left\{k\in\mathbb N_0:\log^{\circ k}m\le C_*\right\},
\]
where \(\log^{\circ0}m=m\), \(\log^{\circ1}m=\log m\), and
\(\log^{\circ k}\) denotes the \(k\)-fold composition of \(\log\) with itself.  For
\(1\le m\le C_*\) set \(\log^*m=0\).  Thus \(\log^*m\) is the number of logarithms
needed to bring \(m\) below the fixed threshold \(C_*\).

For example, if
\[
 m=\exp(\exp(\exp C_*)),
\]
then
\[
 \log m=\exp(\exp C_*),
 \qquad
 \log\log m=\exp C_*,
 \qquad
 \log^{\circ3}m=C_*.
\]
Hence \(\log^*m=3\).  This illustrates the extremely slow growth of \(\log^*m\):
adding one level to the exponential tower increases \(\log^*m\) by only one.

Now define
\begin{align}\label{eq-logstar-weight-new}
 L_m:=(1+\log^*m)^\gamma,
 \qquad \gamma>0.
\end{align}
Indeed, Stirling's formula in the definition
of \(\Omega_{L,A}\) gives the one-dimensional optimization
\[
        \Omega_{L,A}(R)\simeq_A R L_R .
\]
For the iterated-logarithm weight this follows from the slow-variation property
\[
        L_{CRL_R}\simeq_A L_R,
        \qquad R\ge2.
\]
Hence the transferred weight satisfies
\[
        \widetilde L_m\simeq_A L_{\lfloor c_A\log(m+e)\rfloor}.
\]
But \( \log^*(\log(m+e))=\log^*(m+e)-1+O(1)\) and replacing \(\log(m+e)\) by \(c_A\log(m+e)\) changes \(\log^*\) only by a
bounded additive amount.  Therefore,
\[
        \widetilde L_m
        \simeq_A
        L_{\lfloor c_A\log(m+e)\rfloor}
        \simeq_A
        L_m .
\]
Thus the weight \eqref{eq-logstar-weight-new} is logarithmically stable. The
corresponding class \(\mathcal U_L^p\) consists of functions satisfying, for some
\(A,C>0\),
\[
 \|D^\alpha f\|_{L^p}
 \le
 C A^{|\alpha|}
 \frac{|\alpha|!}{(1+\log^*|\alpha|)^{\gamma|\alpha|}},
\]
with the usual harmless convention for \(|\alpha|\le1\).  This class is much weaker
than the class with denominator \((\log |\alpha|)^{\gamma|\alpha|}\), but unlike the
latter it is invariant under the fractional equation.
\end{example}

\begin{proof}[Proof of Theorem \ref{thm-concrete-invariant-class}]
Take \(L_m=(1+\log^*(m+e))^\gamma\). By Example~\ref{ex-log-star-weight},
\(L\) is transfer-admissible and logarithmically stable. The conclusion follows
from Corollary~\ref{cor-log-stable}.
\end{proof}

Let \(\mathfrak L_{\mathrm{inv}}\) be the collection of transfer-admissible weights
satisfying the logarithmic stability condition \eqref{eq-log-stability-new}.  For
\(1\le p\le\infty\) define
\[
 \mathcal U^p_{\mathrm{inv}}
 :=
 \bigcup_{L\in\mathfrak L_{\mathrm{inv}}}\mathcal U_L^p.
\]

\begin{theorem}[Invariant family generated by logarithmically stable weights]
\label{thm-invariant-family}
Let \(s>0\), \(1\le p\le\infty\) and let \(u\in L^p(\mathbb R^n)\) solve
\[
 \Lambda^s u=Vu,
 \qquad
 \|u\|_{L^p}=1.
\]
If \(V\in\mathcal U_{\mathrm{inv}}^\infty\), then
\[
 u\in\mathcal U_{\mathrm{inv}}^p.
\]
Moreover, every logarithmically stable transfer-admissible class
\(\mathcal U_L^p\) is invariant under the fractional pure-potential equation.
\end{theorem}

\begin{proof}
If \(V\in\mathcal U_{\mathrm{inv}}^\infty\), then \(V\in\mathcal U_L^\infty\) for some
\(L\in\mathfrak L_{\mathrm{inv}}\).  By Corollary \ref{cor-log-stable}, the solution
belongs to \(\mathcal U_L^p\), hence to \(\mathcal U_{\mathrm{inv}}^p\).
\end{proof}

\section*{Acknowledgements}
 H.D. was partially supported by the NSF under agreement DMS-2350129. Y.H. was partially supported by the National Key R\&D Program of China (2022YFA1006800), the National Natural Science Foundation of China (Grant No.~12471205), the Science and Technology 
Innovation Program of Hunan Province (2025RC1005), the Hunan Basic Science Research Center for Mathematical Analysis (2024JC2002), and the Key Project of the Hunan Provincial Education Department (24A0004). M.W. was partially supported by the National Natural Science Foundation of China (No.~12571260), the Natural Science Foundation of Hunan Province (No.~2026JJ20012), and the Hunan Basic Science Research Center for Mathematical Analysis (2024JC2002).

\normalem


\begin{thebibliography}{99}

\bibitem{ApraizEscauriaza14}
J. Apraiz, L. Escauriaza, G. Wang, and C. Zhang,
\newblock Observability inequalities and measurable sets,
\newblock \emph{J. Eur. Math. Soc. (JEMS)} \textbf{16} (2014), no. 11, 2433--2475.

\bibitem{BaeBiswas15}
H. Bae and A. Biswas,
\newblock Gevrey regularity for a class of dissipative equations with analytic nonlinearity,
\newblock \emph{Methods Appl. Anal.} \textbf{22} (2015), no. 4, 377--408.

\bibitem{BaeBiswasTadmor12}
H. Bae, A. Biswas, and E. Tadmor,
\newblock Analyticity and decay estimates of the Navier--Stokes equations in critical Besov spaces,
\newblock \emph{Arch. Ration. Mech. Anal.} \textbf{205} (2012), no. 3, 963--991.

\bibitem{BeauchardJamingPravdaStarov21}
K. Beauchard, P. Jaming, and K. Pravda-Starov,
\newblock Spectral estimates for finite combinations of Hermite functions and
null-controllability of hypoelliptic quadratic equations,
\newblock \emph{Studia Math.} \textbf{260} (2021), 1--43.

\bibitem{Biswas12}
A. Biswas,
\newblock Gevrey regularity for a class of dissipative equations with applications to decay,
\newblock \emph{J. Differential Equations} \textbf{253} (2012), no. 10, 2739--2764.


\bibitem{BiswasMartinezSilva15}
A. Biswas, V. R. Martinez, and P. Silva,
\newblock On Gevrey regularity of the supercritical SQG equation in critical Besov spaces,
\newblock \emph{J. Funct. Anal.} \textbf{269} (2015), no. 10, 3083--3119.

\bibitem{Bourgain05}
J. Bourgain,
\newblock \emph{Green's Function Estimates for Lattice Schr\"odinger Operators and Applications},
\newblock Annals of Mathematics Studies, vol. 158,
\newblock Princeton University Press, Princeton, NJ, 2005.

\bibitem{BradshawGrujicKukavica15}
Z. Bradshaw, Z. Gruji\'c, and I. Kukavica,
\newblock Local analyticity radii of solutions to the 3D Navier--Stokes equations with locally analytic forcing,
\newblock \emph{J. Differential Equations} \textbf{259} (2015), no. 8, 3955--3975.



\bibitem{CappielloToft2017}
M. Cappiello and J. Toft,
\newblock Pseudo-differential operators in a Gelfand--Shilov setting,
\newblock \emph{Math. Nachr.} 290 (2017), 738--755.

\bibitem{DjakovMityaginEntire2003}
P. Djakov and B. Mityagin,
\newblock Spectral gaps of the periodic Schr\"odinger operator when its
potential is an entire function,
\newblock \emph{Adv. in Appl. Math.} 31 (2003), 562--596.

\bibitem{DongKim12}
H. Dong and D. Kim,
\newblock On \(L_p\)-estimates for a class of non-local elliptic equations,
\newblock \emph{J. Funct. Anal.} \textbf{262} (2012), no.~3, 1166--1199.

\bibitem{DongRyu26}
H. Dong and J. Ryu,
\newblock \(L_p\)-estimates for nonlocal equations with general L\'evy measures,
\newblock arXiv:2512.24704, 2025.

\bibitem{DongWang24}
H. Dong and M. Wang,
\newblock Log-type ultra-analyticity of elliptic equations with gradient terms,
\newblock \emph{SIAM J. Math. Anal.} \textbf{57} (2025), no. 4, 4609--4630.

\bibitem{DonnellyFefferman88}
H. Donnelly and C. Fefferman,
\newblock Nodal sets of eigenfunctions on Riemannian manifolds,
\newblock \emph{Invent. Math.} \textbf{93} (1988), no. 1, 161--183.

\bibitem{EscauriazaMontanerZhang15}
L. Escauriaza, S. Montaner, and C. Zhang,
\newblock Observation from measurable sets for parabolic analytic evolutions and applications,
\newblock \emph{J. Math. Pures Appl. (9)} \textbf{104} (2015), no. 5, 837--867.

\bibitem{FoiasTemam89}
C. Foias and R. Temam,
\newblock Gevrey class regularity for the solutions of the Navier--Stokes equations,
\newblock \emph{J. Functional Analysis} \textbf{87} (1989), no. 2, 359--369.

\bibitem{Hardy33}
G. H. Hardy,
\newblock A theorem concerning Fourier transforms,
\newblock \emph{J. London Math. Soc.} \textbf{8} (1933), 227--231.

\bibitem{Hoermander85}
L. H\"ormander,
\newblock \emph{The Analysis of Linear Partial Differential Operators III: Pseudo-Differential Operators},
\newblock Springer-Verlag, Berlin, 1985.

\bibitem{KarpeshinaParnovskiShterenberg26}
Y. Karpeshina, L. Parnovski, and R. Shterenberg,
\newblock Bethe--Sommerfeld conjecture and absolutely continuous spectrum of
multi-dimensional quasi-periodic Schr\"odinger operators,
\newblock \emph{Ann. of Math. (2)} \textbf{203} (2026), no. 2, 349--470.

\bibitem{KukavicaVicol09}
I. Kukavica and V. Vicol,
\newblock On the radius of analyticity of solutions to the three-dimensional Euler equations,
\newblock \emph{Proc. Amer. Math. Soc.} \textbf{137} (2009), no. 2, 669--677.

\bibitem{Li24}
D. Li,
\newblock Optimal Gevrey regularity for supercritical quasi-geostrophic equations,
\newblock \emph{Comm. Math. Phys.} \textbf{405} (2024), article no. 30.

\bibitem{LiWang21}
Z. Li and M. Wang,
\newblock Observability inequality at two time points for KdV equations,
\newblock \emph{SIAM J. Math. Anal.} \textbf{53} (2021), 1944--1957.

\bibitem{Lopes2017GelfandShilov}
P. T. P. Lopes,
\newblock Gelfand--Shilov regularity of SG boundary value problems,
\newblock \emph{J. Pseudo-Differ. Oper. Appl.} 8 (2017), no.~1, 55--81.

\bibitem{MorreyNirenberg57}
C. B. Morrey and L. Nirenberg,
\newblock On the analyticity of the solutions of linear elliptic systems of partial differential equations,
\newblock \emph{Comm. Pure Appl. Math.} \textbf{10} (1957), 271--290.

\bibitem{PaicuVicol11} M. Paicu and V. Vicol, \newblock Analyticity and Gevrey-class regularity for the second-grade fluid equations, \newblock \emph{J. Math. Fluid Mech.} \textbf{13} (2011), 533--555.

\bibitem{Poschel2011HillWeighted}
J. P\"oschel,
\newblock Hill's potentials in weighted Sobolev spaces and their spectral gaps,
\newblock \emph{Math. Ann.} 349 (2011), no.~2, 433--458.

\bibitem{Rodino93}
L. Rodino,
\newblock \emph{Linear Partial Differential Operators in Gevrey Spaces},
\newblock World Scientific, Singapore, 1993.

\bibitem{Trubowitz1977}
E. Trubowitz,
\newblock The inverse problem for periodic potentials,
\newblock \emph{Comm. Pure Appl. Math.} 30 (1977), 321--342.

\bibitem{WangWangZhang19}
G. Wang, M. Wang, and Y. Zhang,
\newblock Observability and unique continuation inequalities for the Schr\"odinger equation,
\newblock \emph{J. Eur. Math. Soc. (JEMS)} \textbf{21} (2019), 3513--3572.

\bibitem{WangZhang23}
M. Wang and C. Zhang,
\newblock Analyticity and observability for fractional order parabolic equations in the whole space,
\newblock \emph{ESAIM Control Optim. Calc. Var.} \textbf{29} (2023), Paper No. 63, 22 pp.

\end{thebibliography}
\end{document}